\documentclass[final]{article}

\usepackage{graphicx}
\usepackage{authblk}
\usepackage{amssymb}
\usepackage{epsfig}
\usepackage[sans]{dsfont}
\usepackage[applemac]{inputenc}
\usepackage[english]{babel}
\usepackage{latexsym}
\usepackage{subfigure} 
\usepackage{color}
\usepackage{float}
\usepackage{amsmath}
\usepackage{amsfonts}
\numberwithin{equation}{section}
\usepackage{amsthm}
\usepackage[numbers]{natbib}
\usepackage[bookmarks=true,colorlinks=true,linkcolor={blue},urlcolor={blue}, citecolor={blue},pdfstartview={XYZ null null 1.22}]{hyperref}%

\newcommand{\done}{ {\bf d}_1 }
\newcommand{\R}{\mathbb{R}}
\def\RR{\mathbb R}

\renewcommand{\P}{\mathcal{P}}
\newcommand{\X}{\mathcal{X}}

\newtheorem{proposition}{Proposition}
\newtheorem{theorem}{Theorem}
\newtheorem{definition}{Definition}
\newtheorem{lemma}{Lemma}

\newtheorem{example}{Example}
\newtheorem{assumption}{Assumption}

\def\be{\begin{equation}}
\def\ee{\end{equation}}
\def\bea{\begin{eqnarray}}
\def\eea{\end{eqnarray}}

\begin{document}

%\markboth{Michael Herty \& Mattia Zanella}{Performance bounds for the mean-field limit of constrained dynamics}

%\catchline{}{}{}{}{}

\title{Performance bounds for the mean-field limit\\ of constrained dynamics}

\author{Michael Herty \thanks{Department of Mathematics, IGPM, RWTH Aachen University, Templergraben 55
Aachen, 52062, Germany  \texttt{herty@igpm.rwth-aachen.de}
} \qquad Mattia Zanella \thanks{Department of Mathematics and Computer Science, University of Ferrara, Via N. Machiavelli 35 Ferrara, 44121, Italy \texttt{mattia.zanella@unife.it}
}}

\maketitle

%\begin{history}
%\received{(Day Month Year)}
%%\revised{(Day Month Year)}
%%\accepted{(Day Month Year)}
%%\comby{(xxxxxxxxxx)}
%\end{history}

\begin{abstract}
In this work we are interested in the mean-field formulation of kinetic models under control actions where the control is formulated through a model predictive control strategy (MPC) with varying horizon. The relation between the (usually hard to compute) optimal control and the MPC approach is investigated theoretically in the mean-field limit. We establish a computable and provable bound on the difference in the cost functional for MPC controlled and optimal controlled system dynamics in the mean-field limit. The result of the present work extends previous findings for systems of ordinary differential equations. Numerical results in the mean-field setting are given.
\end{abstract}

%\keywords{Model Predictive Control; mean-field limit; opinion formation models.}

%\ccode{AMS Subject Classification: 35Q93, 91A23, 82C40; 92D25.}

\section{Introduction}	

In recent years many mathematical models of self-organized systems of interacting agents have been introduced in the  literature, see for example \cite{ArmbrusterRinghofer2005aa, BellomoAjmoneMarsanTosin2013aa, BellomoSoler2012aa, CordierPareschiToscani2005aa, CouzinKrauseFranks2005aa, CristianiPiccoliTosin2014,CuckerSmale2007aa, DegondHertyLiu2014aa,
DegondLiuMotsch2013aa, GalamGefenShapir1982aa, HertyPareschi2010aa, HertyRinghofer2011ad, HertyRinghofer2011ac, MotschTadmor2013aa, PareschiToscani2013ab, Toscani2006aa} and the references therein.
The general setting consists of a microscopic dynamics described by systems of ordinary differential equations where the evolution of the state of each agent is influenced by the collective behavior of all  other agents. Examples in those microscopic interacting systems are frequently seen in the real world like: schools of fish, swarm of bees, herds of sheep
and opinion formation in crowds. Of interest is usually the case when the number of agents becomes very large. Here, the qualitative behavior is  studied through a different level of description, i.e. through the introduction of distribution functions whose behavior is governed by kinetic (or fluid--dynamics) partial differential equations. 
\par 
The control mechanisms of self--organized systems has been investigated recently as follow--up questions to the progress in mathematical modelling and simulation. The control of emergent behavior has been studied on the level of the microscopic agents \cite{AlbiPareschiZanella2015bb,AlbiPareschiZanella2015aa} as well as on the level of the kinetic  \cite{AlbiHertyPareschi2014ab,AlbiPareschiZanella2014aa,HertySteffensenPareschi2015aa}
or fluid--dynamic equations \cite{BensoussanFrehseYam2013aa,ColomboPogodaev2012aa,DegondLiuRinghofer2014aa,FornasierSolombrino2013aa}.  The contributions
have to be further distinguished depending on the type of applied control. Without intending to review all literature we give some references on certain classes of control, e.g.,  
sparse control \cite{FornasierPiccoliRossi2014aa}, Nash equilibrium  control   \cite{LasryLions2007aa}, control using linearized dynamics and Riccati equations \cite{HertyRinghofer2011ac,HertySteffensenPareschi2015aa}
or control driven by other external dynamics \cite{AlbiPareschiZanella2014aa,DuringMarkowichPietschmann2009aa}. 
\par 
Here, we focus on a general method to construct a control mechanism, called model predictive control (MPC). MPC utilizes the assumption that agents optimize their  cost functional  {\em not} necessarily over a large time horizon. Instead they determine their (locally best) action by minimizing their cost only over  a short time interval which recedes as time evolves. The methodology of MPC is also called receding horizon control (or instantaneous control when the length of the horizon is equal to one). 
From the modeling point of  view the fact that agents may be able to optimize strategically their trajectories over a small, but finite, interval of time opened several 
connections to socio--economic problems, where each agent, or a portion of them, is influenced in order to force the entire system toward specific patterns. 

MPC has been used in the engineering community for over fifty years, see e.g. \cite{MayneRawlingsRao2000aa,MichalskaMayne1993aa,MichalskaMayne1995aa,Tadmor1992aa} for an overview and further references. However, therein, only a small number of agents $M<\infty$ is  considered and the optimization problems are then studied at the level of ODEs. 
The link between MPC on the level of agents and the MPC on the level of kinetic and fluid--dynamic equations has been subject to recent investigations \cite{AlbiHertyPareschi2014ab,DegondLiuRinghofer2014aa,HertySteffensenPareschi2015aa}, and also the relation between MPC and mean-field games \cite{LasryLions2007aa} has been a subject to recent studies \cite{DegondHertyLiu2014ab}. However, in  {\em all } currently presented  approaches on MPC in relation to mean-field limits the {\em special } case of a receding
time horizon has been considered. While this is computationally advantageous, it is known to have some severe drawbacks: in the case of finitely many agents stability of the controlled system can expected only if the horizon is sufficiently {\em large}, the instability of the controlled system has  been also observed numerically e.g. in \cite{AlbiPareschiZanella2014aa}. 
Further,  MPC  leads to a control that is suboptimal compared 
with the theoretical optimal one, that is a control with infinite control horizon. Except for a very particular case \cite{HertySteffensenPareschi2015aa} 
there is {\em no} result on the relation between the optimal control and the MPC approach in the mean-field limit.

In the case $M<\infty$ there has been recent  progress  on the relation between the time horizon for MPC and 
the stability as well as optimality estimates of MPC controls  \cite{GrimmMessinaTuna2005aa,Grune2009aa,GrunePannekSeehafer2010aa,JadbabaieHauser2005aa}.
In particular an estimate on the difference between MPC 
and optimal control has been given in \cite[Corollary 4.5]{Grune2009aa}.
 The theory therein covers finite and infinite dimensional phase spaces, but still requires the number $M$ of agents to be {\em finite}.  

The main purposes of the present work is to extend the theory presented in \cite{Grune2009aa} to the limit case of infinitely many agents. The goal is to derive the corresponding mean-field results for the optimality estimates under the {\em same} assumptions as in the case $M<\infty.$  While the presentation will cover a general dynamics we exemplify the results on a first--order alignment model, as an extension to models recently presented \cite{AlbiHertyPareschi2014ab,AlbiPareschiZanella2014aa,DegondHertyLiu2014aa}

The rest of the manuscript is organized as follows. First, in Section \ref{sec:example_motivation},  we introduce some notations and results for an exemplified constrained model deriving its mean-field formulation and highlighting the main features of the performance estimate for the MPC approach. In a more general setting, in Section \ref{sec3}, we define the objects of a mean-field optimal control problem subject to a given dynamics proving several estimates in relation to MPC. Here, an example is proposed 
with numerical results, in Section \ref{num},  confirming the theoretical analysis. In Appendix \ref{app:A} and \ref{sec:mean-field} we recall technical details.

\section{Notation and motivating example}\label{sec:example_motivation}
We introduce the notation and exemplify the results obtained in Section \ref{sec3}   
on a simple alignment model \cite{AlbiHertyPareschi2014ab,AlbiPareschiZanella2015aa,CuckerSmale2007aa,PareschiToscani2006aa}.
Let us assume that $M>0$ agents fulfill the dynamics 
\begin{equation}\begin{split}
\dot{x}_i(t)&=\dfrac{1}{M}\sum_{j=1}^M P(x_j(t)-x_i(t)) + u(t), \; i=1,\dots, M, \; t \geq 0,
\end{split} \label{eq:mpc00} \end{equation}
where  $P \geq0 $ is a general interaction function that
 may also depend on variables $\left(x_j\right)_{i=1}^M$  and $x_i=x_i(t) \in \R$ is the state of the $i$th agent at time $t\ge0 $. We denote by 
\begin{equation}\label{eq:fullset}
X(t)=(x_i(t))_{i=1}^{M}, \qquad X_{-i}(t) = (x_j(t))_{j=1, j \not = i }^{M}
\end{equation}
the state of the full system at time $t$ and the state of the all
agents except the $i$th agent, respectively. In the following we will drop the dependence on $t$ whenever the intention is clear. Moreover we assume that initial 
conditions  $X(0)=X_{0}$ are
given. 

The control $u$ is to be determined in order to minimize a given cost functional 
\begin{equation}\label{generalcost}
J^{u}_\infty( X_{0} ) := \int_0^\infty \ell( X(t), u(t) ) dt, 
\end{equation}
where $X(t)$ is the solution to \eqref{eq:mpc00} for the control $u(t)\in U$, with $U\subset\RR$ bounded, and  
initial datum $X(0)=X_0$. In \eqref{generalcost} we introduce a general function $\ell:\R^M\times U\rightarrow \RR $.
Hence, the functional $J$ depends on the initial datum $X_0$ 
as well as the choice of the control $u.$ The dependence of $J$ on the  time horizon is indicated by a subscript $+\infty$, whereas the dependence on the control by the superscript $u$. We assume there exists a solution $u^{*}$ of the optimal control problem 
\begin{equation} \label{cont pb sec1} u^{*} = \textrm{arg}\min_{u}  J^{u}_\infty( X_0 )
\end{equation}
characterized by the Pontryagins Maximum Principle. From the computational point of view this approach is generally too expensive, therefore, a suboptimal approach named model predictive control (MPC) has been proposed. In its simplest form MPC introduces a discretization
in time as follows: let  $t^n=n\Delta t$ with $n=0,\dots,$ and $\Delta t>0$ and set 
$x_{i,n}=x_i(t_n)$ and in analogy to the previous notations $X_n=(x_{i,n})_{i=1}^{M}$. 
(Single-step) MPC with receding time horizon $N$ applies a control $u$ of the type
\begin{equation}\label{uMPC} 
u^{MPC}(t) = \sum\limits_{n=0}^\infty u_n^{MPC} \chi_{[t^{n},t^{n+1})}(t).
\end{equation}
The unknown control actions $u_n^{MPC} \in \R$ are determined at each time $t^{n}$ by 
\begin{equation}\label{MPCsec1} 
u_n^{MPC} =  v_1
\end{equation}
where $(v_k)_{k=1}^{N}$ are the solutions of the following auxiliary minimization problem
\begin{equation}\label{eq:mf 1b}
\left(v_k\right)_{k=1,\dots,N}=\textrm{arg}\min_{ (v_k)_{k=1}^{N} 
}  \Delta t  \; \sum_{k=1}^N \ell( Y_k  ,v_k) \; \mbox{ subject to } \eqref{eq:align_cons:2},
\end{equation}
where the states $Y_k, \; k=1,\dots,N,$ are given by the dynamics \eqref{eq:align_cons:2} for an initial value $X_n$ and a time horizon $N$, i.e., for each $k=1,\dots, N$
\begin{equation}\label{eq:align_cons:2}
y_{i,k+1} = y_{i,k}+\dfrac{\Delta  t}{M}\sum_{j=1}^M P(y_{j,k}-y_{i,k})+\Delta t v_k, \qquad y_{i,1} = x_{i,n}.
\end{equation}
Observe that the discretization of the cost functional $J_\infty^{u}(X_0)$ is now 
\begin{equation}
J_\infty^{u}(X_0) = \sum\limits_{n=0}^{\infty} \ell( X_n, u_n).
\end{equation}
For simplicity we denote its discretized version by the same letter as the continuous functional \eqref{generalcost}. Moreover in the introduced notations the case $N=2$ corresponds to instantaneous control \cite{AlbiHertyPareschi2014ab,AlbiPareschiZanella2014aa,DegondHertyLiu2014ab,DegondLiuRinghofer2014aa}.

A first obvious relation between the optimal control and the control introduced through a model predictive approach is the following
\begin{equation}
J^{u^{MPC}}_\infty(X_0) \geq J^{u*}_\infty(X_0).
\end{equation}
Part of the investigation in \cite{Grune2009aa} is related to a result to establish an {\bf upper bound } on  
$J^{u^{MPC}}_\infty$ by a multiple of $J^{u^*}_\infty$, in particular the result \cite[Theorem 4.2]{Grune2009aa} proves that such a multiplicative factor can be obtained and depends in particular on the optimization horizon $N$ and on the decay rate of the function $\ell(\cdot,\cdot)$. The result of the aforementioned work leads to an estimate at the ODE level of the type
\begin{equation}\label{eq:estimate}
\alpha_N J_\infty^{u^*}(X_0) \leq \alpha_N  J_\infty^{u^{MPC}_N}(X_0) \leq J_\infty^{u^*}(X_0),
\end{equation}
for some $0<\alpha_N \leq 1$. Where we indicated the dependence of $u^{MPC}$ on the time horizon  in problem \eqref{eq:mf 1b} by the subscript $N$ on the control. Further, an 
estimate $\alpha_N  J_\infty^{u^{MPC}_N}(X_0) \leq J_N^{u^{*}}(X_0) $ has been established as an additional result in \cite[Corollary 4.5]{Grune2009aa}. 
Here, $J_N^{u^{*}}$ is defined as in equation \eqref{generalcost} but for a finite time horizon $T=N \Delta t$.
An estimate on the crucial constant $\alpha_N$ is provided e.g. in \cite{GrunePannekSeehafer2010aa}.

We are interested in a corresponding result in the case of a large number of agents, that is in the limit $M\to\infty.$ In fact, according to Theorem \ref{Theorem2.1Card}, the mean-field limit $M\to\infty$ of the dynamics described in \eqref{eq:mpc00} and \eqref{eq:align_cons:2} exists, and formal computations are given in Appendix  \ref{sec:mean-field}. As an example, consider in  the special  function $\ell: \R^{M} \times U \to \R$ 
\begin{equation}
\ell( X, u ) = \frac{1}2 \left( \frac{1}M \sum\limits_{j=1}^M x_j \right)^2 + \frac{\nu}2 u^{2},
\end{equation}
for some regularization parameter $\nu>0.$ Then, the  limit $M\to\infty$ of $\ell$  exits and is
given by 
\begin{equation}
\tilde{\ell}(f,u) = \frac{1}2 \left( \int_{\R} y f_k(y) dy \right)^{2} + \frac{\nu}2 u^{2}
\end{equation}
with $\tilde{l}:\P(\R) \times U \to \R$, where
 $\P(\R)$ denotes the probability measures on $\R$.

Let us consider the dynamics \eqref{eq:align_cons:2}  and denote by $y \to f_k(y)$ the agent probability density at time $t_k$ with $f_k(\cdot) \in \P(\R)$ for $k=1,\dots,N$.  
The limiting equation corresponding to the microscopic dynamics in \eqref{eq:align_cons:2} for $M\to\infty$ and a.e. $y \in \R$ reads  
\begin{equation}\label{eq:mf 2a}
f_{k+1}(y) = f_k(y) - \Delta t \partial_y \int_{\X} P \left( z-y \right) f_k(z) f_k(y) dz -\Delta t v_k  \partial_y  f_k(y), \quad f_1(y) = h_n(y).
\end{equation}
The probability distribution  $h_n(\cdot) \in \P(\R)$ is the distribution $h(t_n)$ at time $t^{n}$ obtained by propagation of the mean-field limit of the original dynamics \eqref{eq:mpc00}, i.e. for each $t \geq 0$
\begin{equation}\label{kinetic eq} \partial_t h(t,x) + \partial_x \left( \int_{\X} P (z-x) h(t,z)h(t,x) dz - u(t) h(t,x) \right)=0. \end{equation} 
In equation \eqref{kinetic eq} $u(\cdot) = u^{MPC}(\cdot)$ is the control obtained by the MPC approximation \eqref{uMPC}
and \eqref{MPCsec1}. The initial state $h(t,0)=h_0(x)$ is obtained as the  probability distribution corresponding to 
the mean-field limit of  the initial data to \eqref{eq:mpc00}. The control $(v_k)_{k=1}^N$ in equation \eqref{MPCsec1} is determined 
by solving the corresponding mean-field optimization problem, i.e., 
\begin{equation}\label{eq:mf 2b}
(v_k)_{k=1,\dots,N} = \textrm{arg}\min_{ (v_k)_{k=1}^{N} 
} \Delta t \; \sum_{k=1}^N \tilde{\ell}( f_k(\cdot)  ,v_k), \mbox{ subject to } \eqref{eq:mf 2a}.
\end{equation}

As usual, the discrete dynamics is recovered by substituting the discrete measure 
$m^M_{X_i}$ in the weak form of the equation. Here,   $\delta$ denotes the Dirac-$\delta$ measure 
and  $m^M_\xi \in \mathcal{P}(\R)$ is defined by 
\begin{equation}
m^M_\xi(x) = \frac{1}M \sum\limits_{i=1}^M \delta(x-\xi_i).
\end{equation}
We refer to \cite{CarrilloFornasierToscani2010aa,DiFrancescoRosiniArchive2015aa,MotschTadmor2011aa}
for rigorous results and more details on the mean-field limit.  As an example, 
note that the mean-field limit  $J_\infty^{u}: \P(\R) \to \R$  is 
\begin{equation}
\tilde{J}^{u}_\infty(h_0) = \int_0^{\infty} \tilde{\ell}( h(t,\cdot), u(t) ) dt,
\end{equation}
where $h$ is determined by equation \eqref{kinetic eq} with initial condition $h_0 \in \P(\R).$  
As before a horizon of $N=\infty$, corresponding to the optimal case, is desirable but computationally inefficient. 
In the sequel we want to establish the estimate \eqref{eq:estimate} also for
the mean-field cost functional $\tilde{J}^{u}.$ Except for the assumptions required
to derive the mean-field limit we only enforce 
the assumptions of \cite[Theorem 4.2]{Grune2009aa} and we will show how
those are sufficient to derive the corresponding estimates. Also, we will justify by obtaining the 
suitable meanfield limits the previously outlined recipe for MPC meanfield control for a broader
class of agent dynamics.

\section{ Optimality estimate for the mean-field cost functional using MPC approach }\label{sec3}
We will follow the approach described in \cite{Grune2009aa,GrunePannekSeehafer2010aa} with applications to the infinite dimensional mean-field case 
taking first into account a discretized system of ordinary differential equations. 
\par 
Let us consider a homogeneous time discretization for $\dot{x}_i = g(x_i(t), X_{-i}(t)) + u(t)$ given by 
\begin{equation}\label{eq:dynamic_general}
x_{i,n+1}=x_{i,n} + \Delta t g\left(x_{i,n},X_{-i,n}\right) + \Delta t u_n
\end{equation}
where $g:\R^{M}  \to \R$ is a general differentiable function that depends on the state of the $i$th agent and
 on the states of other  agents\eqref{eq:fullset}.  Also, $\Delta t=t^{n+1}-t^n>0$ and for simplicity we assume $\Delta t = 1$. Let us suppose that $g$ fulfills the assumptions 
of \cite[Section 4]{BlanchetCarlier2014aa}, see also Appendix \ref{app:A}. In order to pass to the mean-field limit 
we require that each agent trajectory $x_{i,n}$ belong to a compact subset $\X$ of $\R$ for all $n.$ 
Let $U,\X$ be  compact subset of $\R.$ Then, we assume that for $x_{i,0}\in \X$ and $u_n \in U$ we have $x_{i,n}\in X$ 
for each $i=1,\dots,M.$  Then, according to \ref{Theorem2.1Card}  
there exists a function $$G:\X \times \P(\X) \to \R$$  such that the sequence $$G_M(x_{i,n},m^{M}_{X_{-i,n}}) = g(x_{i,n},X_{-i,n})$$ converges
toward $G$ in the limit $M\to \infty$. For the precise definition of $(G_M)_M$ with $G_M:\X\times \P(\X) \to \R$ we refer to equation \eqref{def-glimit}. 
This allows to obtain that the particle density function  $f_n\in\mathcal{P}(\X)$ satisfies the semi--discrete
partial differential equation in strong form
\begin{equation}\label{eq:mf_control}
f_{n+1}(x) = f_n(x) - \partial_x [G(x,f_n(x))f_n(x)] - \partial_x [f_n(x) \; u_n],
\end{equation}
for a given initial distribution $f(0)=f_0\in\mathcal{P}(\X)$. We denote the set of admissible control sequences $(u_n)_{n\in \mathbb{N}_0}$ 
with   with $u_n\in U \subset \R$   by 
 $\mathcal{U}$.   In the following we will always assume that for any given initial distribution $f_0 \in \P(\X)$ and control $u=(u_n)_n$, there exists a sequence of sufficiently regular functions $(f_n)_{n\in\mathbb{N}_0}, f_n\in \mathcal{P}(\X)$, given by the dynamics described in \eqref{eq:mf_control}. This sequence depends on the initial distribution $f_0$ and on the choice of the control sequence  $u=(u_n)_n$.  
% We also define the mean-field cost functional depending on the time horizon (as subindex) and the control $u=(u_n)$ as super-index. 

\begin{definition}
The {infinite horizon mean-field cost} $J_{\infty}^u:\mathcal{P}(\X)\rightarrow \RR^+_0$ is denoted by  
\begin{equation}\label{general mf cost}
J_{\infty}^u(f_0)=\sum_{n=0}^{+\infty}\ell(f_n,u_n),
\end{equation}
where $l :\mathcal{P}(\X)\times U\rightarrow \mathbb{R}^+_0$ is the running cost function  and 
where $(f_n)_n, f_n:\P(\X)\to\R,$ is the solution to equation \eqref{eq:mf_control} with 
initial ditribution $f_0 \in \P(\X)$ and given control sequence $u=(u_n)_n$. 
\end{definition}

\begin{example} 
Consider the discrete problem \eqref{eq:align_cons:2}. Let the cost functional be given by a discretization of \eqref{generalcost} with $\ell$ as in the previous section: 
\begin{equation*}\label{eq:LX_nu_n}
\ell( X_n ,u_n)=\dfrac{1}{2}\left(\dfrac{1}{M}\sum_{i=1}^M x_{i,n}\right)^2+\dfrac{\nu}{2} u_n^2
\end{equation*}
for some fixed parameter $\nu>0.$ The function $\ell$ is symmetric in $X_n.$ Provided that $x_{i,n} \in \X, u_n \in U$, 
 we obtain that $\ell$ is uniformly bounded independently on $M$, i.e. $ \| \ell(\cdot,\cdot) \|_\infty \leq C_0.$  Further, $\ell$ is locally Lipschitz-continuous in $X_n$ 
as composition of locally Lipschitz continuous functions. In fact let $x_{i,n},y_{i,n} \in \X$, then we can compute 
\begin{equation*}\begin{split}
\left | \dfrac{1}{2} \left(\dfrac{1}{M}\sum_{i=1}^M x_{i,n}\right)^2+\dfrac{\nu}{2} {u_n}^2-\dfrac{1}{2}\left(\dfrac{1}{M}\sum_{i=1}^M y_{i,n}\right)^2-\dfrac{\nu}{2} {u_n}^2\right | 
 \leq \frac{2C_1}M \left | \sum\limits_{i=1}^M (x_{i,n}-y_{i,n} ) \right |,
\end{split}\end{equation*} 
with $C_1\ge 0$ the Lipschitz constant. 
%and the right hand side converging toward $\left |\int_{\R}  z (m_X^M(z) - m_Y^M(z)) dz \right |$
Therefore, $\ell(\cdot,\cdot)$ fulfills as function of $X$ the assumptions of Theorem \ref{Theorem2.1Card} and its mean-field limit exists and is given by 
\begin{equation}
\ell(f_n,u_n) = \frac{1}2 \left( \int_{\X} x f_n(x) dx \right)^2 + \frac{\nu}2 u_n^2.
\end{equation}
\\
\end{example} 

The previous example shows that the cost functional \eqref{generalcost}  requires strong symmetry assumptions. This is fulfilled for example if it depends on functions of average quantities of the state of the particles. Under the symmetry assumption we  expect to extend the results proposed in \cite{Grune2009aa}. Therefore,  we require in the following that the running cost $\ell$ is symmetric with respect to each agent, that the running costs are uniformly bounded and Lipschitz continuous with respect to the distance ${\bf d}_1$, defined in Appendix \ref{app:A}.

Let us now introduce the notion of optimal value-function, in the mean-field setting, and show a first result.
\begin{definition}\label{def2}
We denote by  $V_\infty: \P(\X) \to \R$ the  \emph{optimal value function}  of  the mean-field control problem \eqref{eq:mf_control}  associated with the 
 infinite horizon cost $J_{\infty}^u(f_0):$ 
\begin{equation}
V_{\infty}(f_0)=\inf_{u\in\mathcal{U}} J^u_{\infty}(f_0).
\end{equation}
\par 
We define the \emph{approximate optimal cost} $J_N^u:\P(\X) \to \R$ with optimization horizon $N$ as 
\begin{equation}\label{receding horizon cost}
J_N^u(f_0)=\sum_{n=0}^{N-1}\ell(f_n,u_n).
\end{equation}
  The approximate value function $V_N(f_0): \P(\X)\to\R$ 
in the case of receding horizon strategy is defined by
\begin{equation}\label{optimal :receding horizon cost}
V_N(f_0)=\inf_{u\in\mathcal{U}}J^u_N(f_0,u).
\end{equation}
\end{definition}

Further we introduce the notion of a feedback law. A feedback law for $M$ agents is a mapping $\mu_M: \X^{M} \to U$. A symmetric feedback law is a feedback law such that for all $X \in \X^{M}:$ $\mu_M(X)=\mu_M( (x_i)_{\sigma(i)} )$ and any permutation $\sigma\in S_M$, with $S_M$ the symmetric group of degree $M$ 
\begin{equation}
\sigma =
\left(
\begin{matrix}
1 & 2 & \dots & M \\
\sigma(1) & \sigma(2) & \dots & \sigma(M)
\end{matrix}
\right).
\end{equation} 
As for the running cost $\ell$, we further assume that the feedback law $\mu_M$ is symmetric, uniformly bounded and Lipschitz continuous
with respect to ${\bf d}_1$.

We now establish an estimate of the type \eqref{eq:estimate} in the mean-field case. 
Note that the result in \cite{Grune2009aa} alreadys covers the case of a cost functional
\eqref{general mf cost} and \eqref{eq:mf_control}. Therefore, our purposes
is to derive the estimate \eqref{eq:estimate} starting from the finite discrete dynamics 
\eqref{eq:dynamic_general} and in the mean-field limit case $M\to\infty.$

\begin{proposition}\label{prop1}
Let us consider a set of $M$ agents which evolve according to the microscopic dynamics \eqref{eq:dynamic_general} with known initial data $(x_{i,0})_{i=1}^M$. 
Consider the functions $\ell_M: {\X}^{M}\to\RR,$ and $\tilde{V}_M: \X^M \times \mathcal{U} \to \RR$,
% fulfilling the assertions of Theorem \eqref{Theorem2.1Card}.  
and a symmetric feedback $\mu_M: \X^{M} \to U$, fulfilling the assertions of Theorem \ref{Theorem2.1Card} and Definition \ref{def2}.
\par 
Assume furthermore that $\tilde{V}_M$ fulfills  for all $X_0\in \X^{M}$ the inequality
\begin{equation}\label{eq:ineq_1}
\tilde{V}_M(X_0)\ge \tilde{V}_M\left(  \left(x_{0,i} + \Delta t \left( g(x_{i,n},X_{-i,n}) + \mu_M(X_0) \right) \right)_{i=1}^{M} \right)+\alpha \ell_M(X_{0}, \mu_M(X_0) ) 
\end{equation}
with $\alpha\in(0,1]$. Then, there exists a function $\tilde{V}: \P(\X) \to \R$ as mean-field
limit of $\tilde V_M$ for $M\to \infty$ 
 such that for all $f\in \P(\X)$  we obtain 
\begin{equation}%\label{ineq:prop1}
\alpha V_{\infty}(f)\le \alpha J^{u}_{\infty}(f)\le \tilde{V}(f).
\end{equation} 
for $u=(u_n)_n, u_n=\mu(f_n),$ where $\mu$ is the mean-field limit of $(\mu_M)_M.$
\end{proposition}
%proof
\begin{proof}
Due to the assertion of Theorem \ref{Theorem2.1Card} we have $\tilde{V}$,  $\ell:\P(\X)\times U\to \R$ and $\mu:\P(\X) \to U$ exist.  Further, we obtain  for $f_0 \in \P(\X)$ 
(as limit for $M\to \infty$ of the sequence $(m^{M}_{X_0})_M$) the corresponding 
inequality for $\tilde{V}$
\begin{equation}\label{ineq:prop1}
\tilde{V}(f_0) \ge \tilde{V}(f_0 - \partial_x[f_0G(x,f_0)] - \partial_x [f_0\; \mu(f_0)]) +\alpha \ell(f_0,\mu(f_0) ).
\end{equation}
In fact for all $i=1,\dots,M$ and all $M$
\begin{equation}\label{pf:dynamics}
x_{1,i}=x_{0,i}+ \Delta t\;g(x_{i,n}, X_{-i,n}) + \Delta t \: \mu_M( X_0 ),
\end{equation}
which corresponds in the mean-field limit to
\begin{equation}
f_1 = f_0 - \partial_x[f_0G(x,f_0)] - \partial_x [f_0\; \mu(f_0)].
\end{equation}
The mean-field limit $\tilde{V}$ is obtained as limit of the 
sequence  ${\bf V}_M: \P(\X)\to \R$ where
\begin{equation}{\bf V}_M( f ) = \inf\limits_{ X \in \X^{M} } \{ V_M(X) + \omega( {\bf d}_1( m^{M}_X, f ) ) \},
\end{equation}
see Theorem \ref{Theorem2.1Card}. We therefore have  
${\bf V}_M( m_{X}^{M} ) = V_M(X)$ and therefore for all $X_0 \in \X^{M}$ 
$$ {\bf V}_M(   m_{X_0}^{M} ) \geq {\bf V}_M( m_{X_1}^{M} ) + \alpha {\bf \ell}_M( m_{X_0}^{M}, \mu_M(m_{X_0}^{M}) ).$$
Further, ${\bf V}_M$ has modulus of continuity $\omega,$ i.e., 
$ | {\bf V}_M( f) - {\bf V}_M( g) | \leq \omega( {\bf d}_1( f,g ) ).$
Let $f_0 \in \P(\X)$ be the limit  of $m^{M}_{X_0}$ for $M\to \infty.$ 
Note that the limit exists for metric ${\bf d}_1$ on the probability measures, 
since $\X$ is compact subset of $\R$ and therefore
 $m^{M}_{X_0}$ has finite 1--Wasserstein distance, i.e., $\int_\X |x| d m^{M}_{X_0} < C$
  with $C$ independent of $M$ and $X_0.$  Due to the dynamics \eqref{pf:dynamics}
we have $f_1$ is then the limit of $m^{M}_{X_1}$, $X_1$ given by \eqref{pf:dynamics}. 
Since ${\bf V}_M$ has    modulus of continuity $\omega,$ we obtain 
$$ {\bf V}_M(   f_0 ) \geq {\bf V}_M( f_1 ) + \alpha {\bf \ell}_M( f_0, \mu_M(f_0) ).$$
Hence, we have 
$$\tilde{V}(f_0) \ge \tilde{V}(f_1) + \alpha \ell(f_0,\mu(f_0)).$$
Define now $u_n = \mu(f_n)$ and consider the solution to \eqref{eq:mf_control}. Since $X_0 \in \X^M$ is arbitrary we obtain that \eqref{ineq:prop1} holds for all $f_0 \in \P(\X)$ 
and therefore  
\begin{equation}
\tilde{V}(f_{n}) \ge \tilde{V}(f_{n+1}) + \alpha \ell(f_n,\mu(f_n)).
\end{equation}
Summation over  $n$ yields
\begin{equation}%\label{eq:ineq_1}
\alpha\sum_{n=0}^{K-1}\ell(f_n,u_n)\le \tilde{V}(f_0)-\tilde{V}(f_K) \leq \tilde{V}(f_0).
\end{equation}
Let now $K\to \infty,$ then $\tilde{V}(f_0)$ is an upper bound for 
$ J^{u}_\infty = \sum_{n=0}^{\infty}\ell(f_n,u_n)$
and where $u_n= \mu(f_n).$ Since $u_n$ is an admissible control we obtain for all $f_0\in \P(\X)$
\begin{equation}
\alpha V_{\infty}(f_0)\le \alpha J^u_{\infty}(f_0)\le \tilde{V}(f_0),
\end{equation}
our assertion as limit of discrete measures.
\end{proof}

The previous results holds for any  family of functions $\tilde{V}_M$ and any symmetric feedback
law. The idea is now to establish the inequality in \eqref{eq:ineq_1} for a general MPC strategy and 
a family of functions $\tilde{V}_M$ given by the optimal running costs $V_N$ as in
Definition \eqref{def2}. In order to establish equation \eqref{eq:ineq_1} for a broad
class of running costs $\ell$, the  functions $\rho,\beta$ have been introduced
in Section 3 in \cite{Grune2009aa}. We recall their definition and assertions in Definition \ref{mhdef} below.
Under Assumption \ref{ass:beta} we prove that 
$\mu=u^{MPC}_N$ and $V_N$ fulfill the assertions of Proposition \ref{prop1}. 
The Assumption \ref{ass:beta} is the mean-field analogous to the assumption imposed in \cite[Assumption 3.1]{Grune2009aa}.

\begin{definition}\label{mhdef}
We  say that a function  $\rho: \RR^+\rightarrow \RR^+$ is of class $\mathcal{K}_{\infty}$ if
\begin{itemize}
 \item[$(i)$]$\rho(0)=0$, 
 \item[$(ii)$] $\rho(\cdot)$ is strictly increasing
 \item[$(iii)$] $\rho(\cdot)$ is unbounded. 
 \end{itemize}
Moreover a continuous function $\beta: \RR^+\times\RR^+\rightarrow \RR^+$ is of class $\mathcal{KL}_0$,  if 
 $\forall r>0$ we have $\displaystyle\lim_{r\rightarrow +\infty}\beta(r,t)=0$ and for each $t\ge 0$ we either have
$\beta(\cdot,t)\in\mathcal{K}_{\infty}$ or (b) $\beta(\cdot,t)\equiv 0.$
\end{definition}
%An example of a function of class $\mathcal{KL}_0$ is $\beta(r,n)=C \sigma^{n} r.$ 
We will denote by $\ell^*(f)$ the minimum of the mean-field running cost $\ell$ and as in \cite{Grune2009aa}
we assume it exists
\begin{equation}\label{ellstar}
\ell^*(f)=\min_{u\in\mathcal{U}} \; \ell(f,u).
\end{equation}

\begin{assumption}[]\label{ass:beta}
We assume that $\ell^*(f)$ is well--defined for all $f \in \P(\X).$ Further, for given $\beta\in\mathcal{KL}_0$ and
 each $f_0\in\mathcal{P}(\X)$ ,there exists a sequence of controls $(u_n)_n, u_n\in\mathcal{U}$ depending only on  $f_0$ such that for each $n$ we have
\begin{equation}
\ell(f_n,u_n)\le \beta(\ell^*(f_0),n).
\end{equation}
\end{assumption}

%Note that for what has been introduced in Section \ref{sec:example_motivation} this assumption is fulfilled provided that $\beta$ is given by equation \eqref{eq:exp_controllability}, see Example \ref{ex2} below. 
 
 In the following Lemma we prove that Assumption 
\ref{ass:beta} is fulfilled provided that the finite--dimensional problem fulfills the 
corresponding assumption \cite[Assumption 3.1]{Grune2009aa}. We establish
the proof in the special case of $\beta$  given by 
\begin{equation}\label{eq:exp_controllability}
\beta(r,n)=C\sigma^n r,
\end{equation}
where $C\ge 1$ is the overshoot constant and $\sigma\in(0,1)$ the decay rate. Clearly, the particular 
choice $\beta(r,n)  \in\mathcal{KL}_0.$

\begin{lemma}\label{lemma:1}
Let $\beta$ be given by equation \eqref{eq:exp_controllability}. Consider a  dynamics with $M$ agents given by the 
dynamics of equation \eqref{eq:dynamic_general} with a control sequence $(u_n)_n$ and $u_n\in\mathcal{U}$
and initial conditions $X_0 \in \X^{M}.$ Assume $ \ell_M: \X^{M} \times \mathcal{U} \to \R$ and $\ell^{*}_M:\X^{M} \to \RR$ 
fulfill the assumptions
of Proposition \ref{prop1} for all $M$.
Further, we assume that \cite[Assumption 3.1]{Grune2009aa} holds, that is for all $M$ we have
\begin{equation}\label{ineq:discrete1}
\ell_M(X_n,u_n)\le\beta(\ell_M^*(X_0),n).
\end{equation}
Then, the mean-field limit $(\ell_M)_M$ and $(\ell^{*}_M)_M$ exist and the limit $\ell:\P(\X)\times \mathcal{U} \to \RR$
and $\ell^{*}:\P(\X)\to \RR,$  fulfills Assumption \ref{ass:beta}:
\begin{equation}\label{ineq:MF}
\ell(f_n,u_n)\le\beta(\ell^*(f_0),n).
\end{equation}
\end{lemma}
\begin{proof}
Due to the assumptions on the family $(\ell_M)_M$ given in Proposition \ref{prop1} we have the existence of the mean-field limit $\ell$
according to Theorem \ref{Theorem2.1Card}. Consider the family of functions
$$\beta_M (X,n) := \beta(\ell^{*}_M(X),n).$$
Clearly, the function $\beta_M$ is symmetric in $X\in \X^{M}.$ Using the definition of $\beta$ by equation 
\eqref{eq:exp_controllability} and the properties of $\ell^{*}_M$ 
we have that  $\beta_M(X,n)$ is uniformly bounded with respect to $X$ on the compact subset $\X^{M}$
by $ C \sigma^{n} \| \ell^{*}_M(X) \|.$ For each $r_1,r_2$ such that $|r_1-r_2|<\delta$ we have 
\[
|\beta(r_1,n)-\beta(r_2,n)|\le C\sigma^n |r_1-r_2|.
\]
Hence, for 
 $\epsilon=C\sigma^{n} \delta$, we have uniform continuity of $\beta_M$ 
due to the uniform continuity of $\ell^{*}_M.$ If $\omega(\cdot)$ is the modulus of continuity
of $\ell^*_M$ then $ C \sigma^{\tilde{n}} |\omega(\cdot)|$ is the modulus of continuity of $\beta_M.$ 
Hence, for each fixed $n$ there exists the mean-field limit ${\bf \beta}$ of $(\beta_M)_M.$ Also, there exists
the mean-field limit $\ell^*$ of $(\ell^*_M).$ Due to the Lipschitz continuity of $\beta$ 
we also have that $\sup_X | \beta(\ell^*_{M_k}(X)) - \beta(\ell^*( m^{M_k}_X ) | \to 0$ for  $(M_k)_k \to \infty.$ 
Therefore, the mean-field limi ${\bf \beta}(f,n) = \beta(\ell^*(f),n).$ Similarly to what we have proven in Proposition \ref{prop1} it follows that the inequality \eqref{ineq:discrete1} implies then \eqref{ineq:MF}.
\end{proof}

\begin{example}\label{ex2}
Consider the example of Section \ref{sec:example_motivation}. The running cost has been given 
by \begin{equation}\label{eq:cost_gen}
\ell(f_n,u_n)=\dfrac{1}{2}\left(\int_{\X}x f_n(x)dx\right)^2+\dfrac{\nu}{2}u_n^2.
\end{equation}
The optimal running cost $\ell^*$ can be computed explicitly and is given by 
\begin{equation}\label{eq:optimal_cost}
\ell^*(f_n)=\dfrac{1}{2}\left(\int_{\X}x f_n(x)dx\right)^2.
\end{equation}
From the mean-field dynamics for $f_n$ are given by \eqref{eq:mf 2a}. Upon integration on $\X$ 
we obtain 
\begin{equation}
 \int_{\X} x f_{n+1}(x) dx = \int_{\X} x f_n(x) dx + \Delta t \; u_n .
 \end{equation}
In \cite{AlbiHertyPareschi2014ab} the following feedback law $\mu: \P(\X) \to \mathcal{U}$ 
 has been proposed as instantaneous MPC: 
 \begin{equation}\label{eqmhh}
  \mu(f_n) = \frac{1}{ 1+\nu } \int_{\X} x f_n(x) dx.
  \end{equation}
Using $\Delta t u_n:= \mu(f_n)$  the 
optimal running cost $\ell^{*}(f_n)$ is expressed in terms of the initial cost $\ell^{*}(f_0)$ 
as 
\begin{equation}
\ell^*(f_n)=\dfrac{1}{2}\left(1-\dfrac{1}{1+\nu}\right)^2 \left(\int_{\Omega}xf_{n-1}dx\right)^2=\left(1-\dfrac{1}{1+\nu}\right)^{2n} \ell^*(f_0).
\end{equation}
Therefore we have
%obtain for $\Delta t u_n:= \mu(f_n)$
\begin{equation}\label{ttemp}
\ell(f_n,u_n)=\left(1+\dfrac{\nu}{(1+\nu)^2}\right)\left(1-\dfrac{1}{1+\nu}\right)^{2n} \ell^*(f_0)=C\sigma^n \ell^*(f_0).
\end{equation}
The overshoot constant $C$ and the decay rate $\sigma$ is computed explicitly for a given regularization $\nu>0$ as 
\begin{equation}
C=1+\dfrac{\nu}{(1+\nu)^2}\ge 1,\qquad \sigma=\left(1-\dfrac{1}{1+\nu}\right)^2\in (0,1).
\end{equation}
Consider the receeding horizon costs with length one as  $\tilde{V}:\P(\X)\to \R$ defined as
\begin{equation}
\tilde{V}(f_0):=\sum_{n=0}^{1} \ell(f_n,\mu(f_n) ). 
\end{equation}
Due to equation \eqref{ttemp} we obtain the assertion of Proposition \ref{prop1} is true by  simple computation
\begin{equation}
 \tilde{V}(f_0) \geq \tilde{V}(f_1) + \alpha \ell(f_0,\mu(f_0))
 \end{equation}
provided that $\alpha:=1 - ( C \sigma )^2$ fulfills $0 < \alpha.$ 
This yields a  bound on the regularization parameter $\nu.$ 
This estimate for $\alpha$  is {\em only} valid in the  case of the feedback law \eqref{eqmhh}.
The idea is to generalize the result to arbitrary symmetric running costs $\ell$ 
and {\em different} control horizons. In the numerical results we then observe for large values of $\nu$ 
also a decay in the receeding horizon costs provided the control horizon is sufficiently large. 
\end{example}

The following Lemma is the analog to \cite[Theorem 4.2]{Grune2009aa}.  
The main idea is to establish the inequality \eqref{eq:ineq_1} using Lemma \ref{lemma:1}
for  a function 
$\tilde{V}$ given by the approximate value function \eqref{optimal :receding horizon cost}. 
The discrete approximate optimal cost $J^{u}_{N,M}:\X^{M} \times \mathcal{U}^{N} \to \R$  with running
cost $\ell_M:\X^{M}\times \mathcal{U} \to \R$  and corresponding approximate value function $V_{N,M}:\X^{M}\to \R$ 
are obtained by considering the discrete measure $m_X^{M}$ for $X \in \X^{M}$ and fixed $M:$
\begin{equation}\begin{split}\label{discrete opt}
{\bf V}_{N,M}(X_0) := V_N( m_{X_0}^{M} ), \;{\bf J}^{u}_{N,M}(X_0, (u_n)_{n=0}^{N-1} ) := \sum\limits_{n=0}^{N-1} \ell_M(X_n,u_n).
\end{split}\end{equation} 
where 
\begin{equation}
\ell_M(X) = \ell( m_{X_n}^{M}, u_n).
\end{equation}
Here, $X_n = ( x_{i,n} )_{i=0}^{M}$ fulfills the discrete dynamics \eqref{pf:dynamics} with initial data $x_i(0)=x_{i,0}.$
 We {\em assume } that  the discrete functions fulfill the corresponding relation \eqref{optimal :receding horizon cost}  for all $X \in \X^{M}:$
$${\bf V}_{N,M}(X) = \min\limits_{ (u_n)  \in \mathcal{U}^{N} } {\bf  J}^{u}_{N,M}(X, (u_n)_{n=0}^{N-1} ).$$
The symmetric feedback law $\mu$ is the MPC feedback introduced on the discrete level by equation 
\eqref{MPCsec1} and equation \eqref{eq:mf 1b}, respectively.

\begin{lemma}\label{th:1}
Consider the discrete dynamics \eqref{eq:dynamic_general} with $M$ agents and $\beta$ given by equation \eqref{eq:exp_controllability} with $C\geq 1$
and $\sigma \in (0,1).$ Consider
a model predictive control horizon of $N.$ Assume the family $(\ell_M)_M, \ell_M:\X^{M}\times \mathcal{U} \to \R$ fulfill the assertions of Proposition \ref{prop1}.  Assume assumption \ref{ass:beta} holds true. Let 
${\bf V}_{N,M}, \ell_M$ and ${\bf J}^{u}_{N,M}$ be given by equation \eqref{discrete opt}. Given are   sequences $\lambda_n>0$,   $n=0,\dots,N-1$ and $\nu>0$ such that 
\begin{eqnarray}
\sum\limits_{n=k}^{N-1}\lambda_n \leq C \lambda_k \frac{1-\sigma^{N-k}}{1-\sigma},  \;k=0,\dots,N-2, \label{4.2} \\
\nu \leq \sum\limits_{n=0}^{j-1} \lambda_{n+1} + C \lambda_{j+1} \frac{1-\sigma^{N-j}}{1-\sigma}, \; j=0,\dots,N. \label{4.1}
\end{eqnarray}
holds true. Assume that then also 
\begin{equation} \label{4.3} \sum_{n=0}^{N-1}\lambda_n-\nu\ge \lambda_0\alpha, \end{equation}
holds true for some $\alpha \in (0,1].$
Then, for any $M$ and any $X_0 \in \X^{M}$ and any running cost $\ell_M$ fulfilling \eqref{ineq:discrete1}
we obtain \eqref{eq:ineq_1} for the MPC feedback law $\mu_M$ given by \eqref{pf:feed}
and for the value function 
$$\tilde{V}_M:={\bf V}_{N,M}.$$ 
Provided that $(\mu_M)_M$ is symmetric and fulfills the assertions of Theorem \ref{Theorem2.1Card}, we obtain for each $f \in \P(\X)$ as limit of $(m_{X}^{M})_M, M\to\infty,$ the inequality
\begin{equation}\label{final}
\alpha V_{\infty}(f)\le \alpha J_{\infty}^{u}(f)\le V_N(f)
\end{equation}
where $u=(u_n)_n, u_n= \mu(f_n)$ and where $\mu$ is the mean-field limit of $(\mu_M)_M.$  
\end{lemma}
{\em Sketch of the proof}.
The proof is analogous to the proof of \cite[Theorem 4.2]{Grune2009aa}. We recall that condition \eqref{4.3} is equivalent to the assertion \cite[(4.3)]{Grune2009aa}. 
For $\beta$ given by equation \eqref{eq:exp_controllability} the assertions \cite[(4.1),(4.2)]{Grune2009aa} simplify to equation \eqref{4.1} and \eqref{4.2}, respectively. 
Consider $M$ agents with corresponding arbitrary initial condition $X_0 \in \X^{M}.$ Consider the finite horizon problem of length $N$ given by
\begin{equation}\label{pf:ex} (u^{*}_n)_{n=0}^{N-1} = \mbox{ arg }\min\limits_{ (u_n)_n \in \mathcal{U}^{M} }  {\bf J}_{N,M}^{u}(X_0, (u_n)_{n=0}^{N-1}).\end{equation}
Then, we denote the corresponding optimal trajectory $X_n^{*}$ obtained through the dynamics \eqref{eq:dynamic_general} for $u_n=u_n^{*}.$ 
We define $$\lambda_{n,M} = \ell_M( X_n^{*}, u_n^{*} ), \; n=0,\dots,N-1$$
and $$\nu_{M} = {\bf V}_{N,M}(X_1^{*}).$$ 
Similarly to  \cite[Proposition 4.1]{Grune2009aa} the values $\lambda_{n,M}$ and $\nu_M$ 
defined in the proof above fulfill equation \eqref{4.1} and equation \eqref{4.2}. This result
has been established in the case of finite number of agents in a sequence of auxiliary aftermaths
that are not repeated here. 
Now, consider the MPC feedback law $\mu_M(X) = v_0$ where  
\begin{equation}\label{pf:feed} (v_0)_{k=0,\dots,N-1} = \mbox{ arg} \min\limits_{ (v_k), v_k \in \mathcal{U} } \sum\limits_{n=0}^{N-1} \ell_M( Y_n, v_k ) \end{equation}
where $Y_n \in \X^{M}$ solves equation \eqref{eq:dynamic_general} with initial data $Y_{0}=X$  and let $(X_n^{\mu})_n$ be the trajectory obtained through \eqref{eq:dynamic_general} for initial data $X_0$ and for $u_n = \mu(X_n)$. 
We observe that $u_0^{*} = \mu(X_0)$ and $X^{\mu}_i=X^{*}_i$ for $i=0$ and $i=1.$ Therefore, 
$\ell_M( X_0, u^*_0) = \ell_M(X_0,\mu(X_0)).$ Therefore, we obtain for all $M$  and any $\alpha$  from equation \eqref{4.3}
\begin{eqnarray*}
 {\bf V}_{N,M}(X^{\mu}_1) + \alpha \ell_M(X_0,\mu(X_0)) =   {\bf V}_{N,M}(X^{*}_1) + \alpha \ell_M(X_0, u_0^{*}) \\
= \nu_M + \alpha \lambda_{0,M} \leq   \sum\limits_{n=0}^{N-1} \lambda_{n,M} =  \sum\limits_{n=0}^{N-1} \ell_M(X_n^*,u_n^*) = {\bf V}_{N,M}(X_0).
\end{eqnarray*}
Therefore, ${\bf V}_{N,M}$ fulfills the assertion on $\tilde{V}_M$ of Proposition \ref{prop1}. 
The second assertion follows as a consequence of Proposition \ref{prop1}. This finishes the outline of the proof. \\

The assumption on existence of an optimal control \eqref{pf:ex}
for ${\bf J}_{N,M}$ is also precisely as in the case of finitely many agents. Note that as in the finite dimensional case the
 optimal control might not exist.  The previous result \eqref{final} gives a {\em performance bound } in the following sense: 
due to the definition of the approximate value function $V_N(f)$ and $V_\infty(f)$ we have 
$$ V_N(f) \leq V_\infty(f).$$
Therefore, we obtain the (usable) estimate on the suboptimality of the MPC $\mu$ as 
\begin{equation}\label{reallyfinal}
 J_{\infty}^{u}(f) \le \dfrac{1}{\alpha} V_{\infty}(f).
\end{equation}
This precisely tells the dependence of the  MPC cost on the optimal expected cost $V_{\infty}$
provided that $\alpha$ is known. The value of $\alpha$ is  the effective degree of $\mu$  with respect to the (unknown) 
infinite horizon control.  Clearly, the computation of $\alpha$ fulfilling inequality \eqref{4.3} is in general 
a difficult task requiring estimates on the value function and running costs.  However, for $\beta$ given by equation \eqref{eq:exp_controllability}
we may estimate  $\alpha$ solely based on the inequalities \eqref{4.2} and \eqref{4.1}. This estimate is denoted by $\alpha_N.$ 
The corresponding result is independent of the meanfield limit and has been established in \cite[Theorem 5.4]{GrunePannekSeehafer2010aa}.
\begin{lemma}\label{th:22}
Let $\beta$ be given by equation \eqref{eq:exp_controllability} for some $C\geq 1$ and $\sigma \in (0,1).$ Let $N$ be the  prediction horizon $N$. Given 
is a sequence $\lambda_n$ and $\nu>0$ such that equation \eqref{4.2} and \eqref{4.1} holds true.  Assume that
\begin{equation}\label{alphaN}
\alpha_N=1-\dfrac{(\gamma_N-1)\displaystyle\prod_{i=2}^N (\gamma_i-1)}{\displaystyle\prod_{i=2}^N\gamma_i-\displaystyle\prod_{i=2}^N(\gamma_i-1)}>0
\end{equation}
holds with $\gamma_i=C\displaystyle\sum_{n=0}^{i-1}\sigma^n$. Then, for $\alpha=\alpha_N$ the inequality  \eqref{4.3} is fulfilled.
\end{lemma}
Equation \eqref{alphaN} is therefore called  performance bound and may be computed a priori to estimate the distance 
of the optimal cost towards the MPC controlled problem. It solely depends on $C$ and $\sigma$ being the estimates on a the running cost $\ell.$ 
As already noted in \cite{Grune2009aa} this estimate might give not necessarily optimal performance bounds.

\section{Numerical Results}\label{num}

First, we investigate the performance bound \eqref{alphaN}. In the example \ref{ex2} 
we have the following explicit values for $C$ and $\sigma$: 
$$C=1+\frac{\nu}{ (1+\nu)^2 },\qquad \sigma= \left(1-\frac{ 1}{ 1+\nu}\right)^2.$$
Estimations on the coefficient $\alpha_N$ allow to measure the quality of the MPC generated control sequence. We depict the value of $\alpha_N$ as a function of $N$ and $\nu$ in Figure \ref{fig:alphaN_nu}. The performance bound can only be used if $\alpha_N>0$ and we indicate the line $\alpha_N=0$ by a black line. We observe that the performance bound increases with respect to the MPC horizon as expected. The best bound is $\alpha_N=1/2$.  For large values of the regularization parameter $\nu$ 
we have to consider a sufficiently large MPC horizon $N$ in order to use the theoretical results. Moreover,
 we observe that the result of Lemma \ref{th:22} is  consistent with the estimate derived in the special case of example \ref{ex2} in the case $N=2.$
 The numerical results below indicate that the bound is too pessimistic, similarly to what has been already observed in the finite dimensional case.

\begin{figure}[htb]
\centering
\includegraphics[scale=0.5]{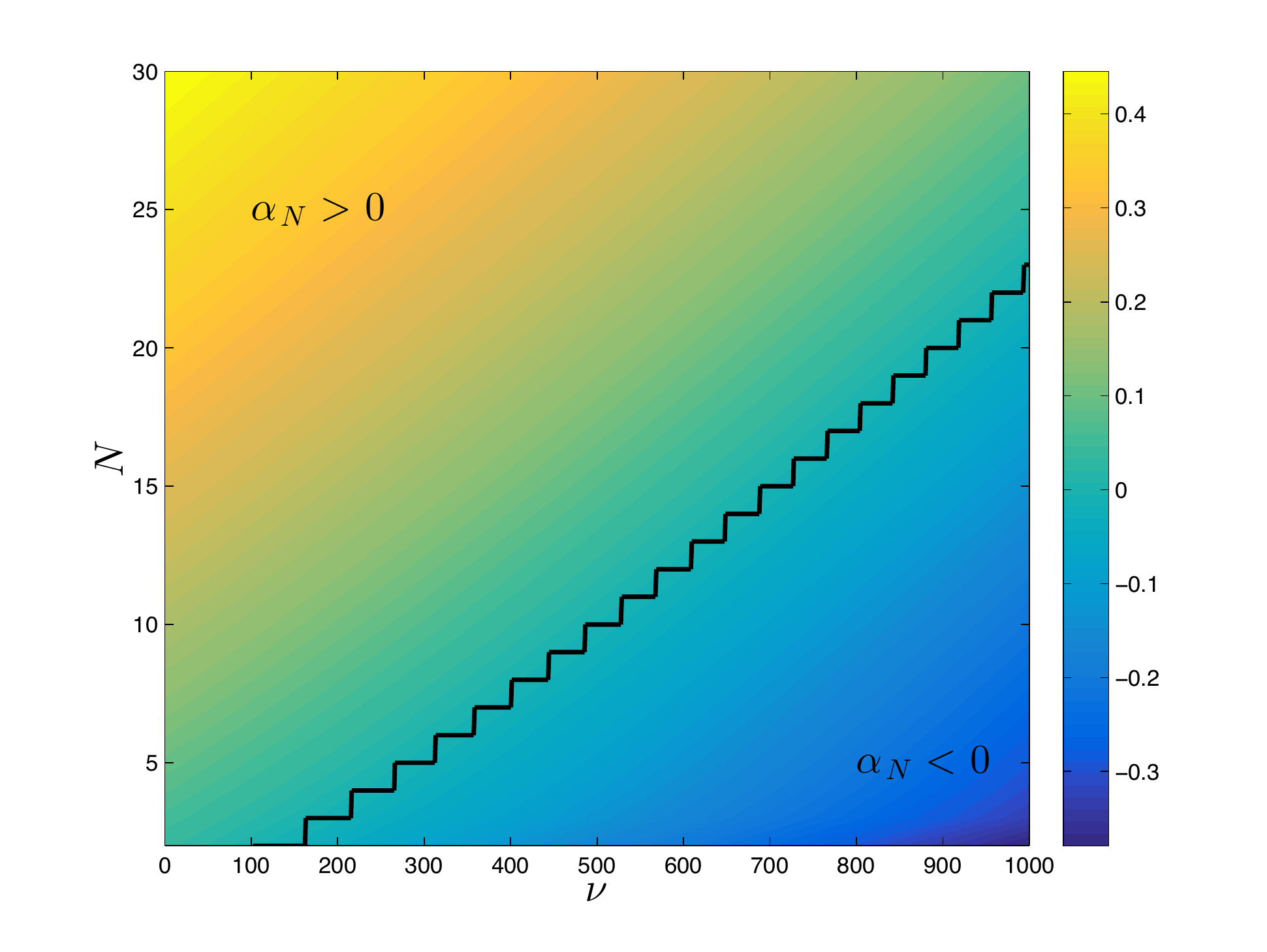}
\caption{Computation of $\alpha_N$ for different values of the regularization parameter $\nu$. We observe for increasing values of $\nu$ 
 corresponding longer control horizons in order to recover positivity of the value $\alpha_N$. }
 \label{fig:alphaN_nu}
\end{figure}

As a numerical example we propose the following discretization coherently with \ref{sec:mean-field}. This discretization reduces the $N$ step MPC problem to again a discrete problem of $M$ agents. We approximate the initial distribution $f_0 \in \mathcal{P}(\X)$ by $f_{M,0}$ given by  a sum of Dirac delta 
\begin{equation}\label{eq:idnuma}
f_{M,0}= \dfrac{1}{M}\sum_{i=1}^M\delta(x-x_{i,0}).
\end{equation}
located at points $x_{i,0} \in \X.$ For the Example \ref{ex2} we observe that if $f_0=f_{M,0},$ then 
$f_n$ is also composed of a sum of Dirac delta. We assume in the following that $f_0$ as well
as $f_n$ decays to zero for $x \to \partial\X$. We observe that if $\int_{\X} f_0 dx=1$ then 
we have $\int_{\X} f_n dx=1$. An approach based on Dirac delta 
converges toward a continuous distribution function in the limit $M\rightarrow +\infty$, 
 provided we have a considerably amount of particles centered in $x_{i,n}\in\X$. 
 Within the described discretization we also recover the setting  of \cite{Grune2009aa,GrunePannekSeehafer2010aa} as numerical scheme. \\
 
Thanks to the structure of example \ref{ex2} further simplifications can be obtained. We recall the  mean-field running cost  
 $\ell(f_n.u_n)=\dfrac{1}{2}\left(\int_{\X}xf_n(x)dx\right)^2+\dfrac{\nu}{2}u_n^2.$ We consider the mean-field equation equivalent to the discretized dynamics of  \eqref{eq:mpc00} for $P=1$ and $\Delta t=1$
\begin{equation}\label{eq:dynamics_num}
f_{n+1}(x)=f_n(x)- \partial_x\int_{\X}(y-x)f_n(y)f_n(x)dy- \partial_x \left( u_nf_n(x) \right),
\end{equation}
a detailed derivation is given in Appendix \ref{sec:mean-field}. Upon multiplication by a general $x\in\X$ and integrating with respect $dx$ we obtain
\begin{equation}\label{eq:mean_evo}
\int_{\X}xf_{n+1}(x)dx=\int_{\X}xf_{n}(x)dx + u_n.
\end{equation}
If we introduce a new variable for the mean $Y_n := \int_{\X}xf^{n}(x)dx$ the problem simplifies to the equation for the evolution of $Y_n$. Further, the cost function is also expressed in terms
of $Y_n$ as 
$
\ell(Y_n,u_n) = \frac{1}2 Y_n^{2} + \frac{\nu}2 u_n^2, 
$
and equation \eqref{eq:mean_evo}
$Y_{n+1}=Y_n + u_n.
$

Using the reformulation of the control of the mean the problem therefore reduces
to a problem appearing in the existing theory \cite{Grune2009aa}. In particular,
the MPC subproblem to determine the optimal control for the horizon $N$ 
is solved explicitly for the previous dynamics. We computed for a horizon $N$
the MPC control at time $n$ and initial data $Y_0$ as  
$(u^{MPC}_N)_n(Y_0)= v_1,$ where 
$$ (v_j)_{j=1}^N := \textrm{arg}\min \sum\limits_{j=n}^{n+N} \ell(Y_j,u_j), \; Y_{j+1}=Y_j + u_j, Y_n=Y_0.$$
For a fixed time horizon $T=100$, fixed initial datum $Y_o$ and  fixed $N$ we then compute the value of the cost functional  for 
$$J_T^{u^{MPC}_N} = \sum\limits_{n=0}^{T} \ell(Y_j,(u^{MPC}_N)_n)$$
where $Y_{j+1}=Y_j+ (u^{MPC}_N)_n(Y_n)$. Further, we compute $J_{100}^{u^{MPC}_N}$ to obtain the optimal cost
 $V_{100}^{*}$. 
 %We compute $\alpha_N$ according to Lemma \ref{th:22}. 
 
 According to Lemma \ref{th:22} we obtain the behavior of the MPC cost $J_T^{u^{MPC}_N}$
 in relation to the optimal cost   $V_{100}^{*}$ in Figure \ref{fig.01}. As expected for larger  MPC horizons we observe convergence towards
 the optimal cost. The performance bound $\alpha_N$ is negative for $N\leq 4$ and therefore the Theorem \ref{th:1} can not be applied. 
  In the results we choose $\nu=10^{2}$. We observe that the bound on $\alpha_N$ is quite pessimistic and the distance of the estimated mismatch of the MPC controlled case to the optimal one is quite large for small horizons, i.e., of order $10^{3}$ for the horizon $N=5$.
 
\begin{figure}[htb]
\begin{center}\label{fig:intro}
\includegraphics[scale=0.5]{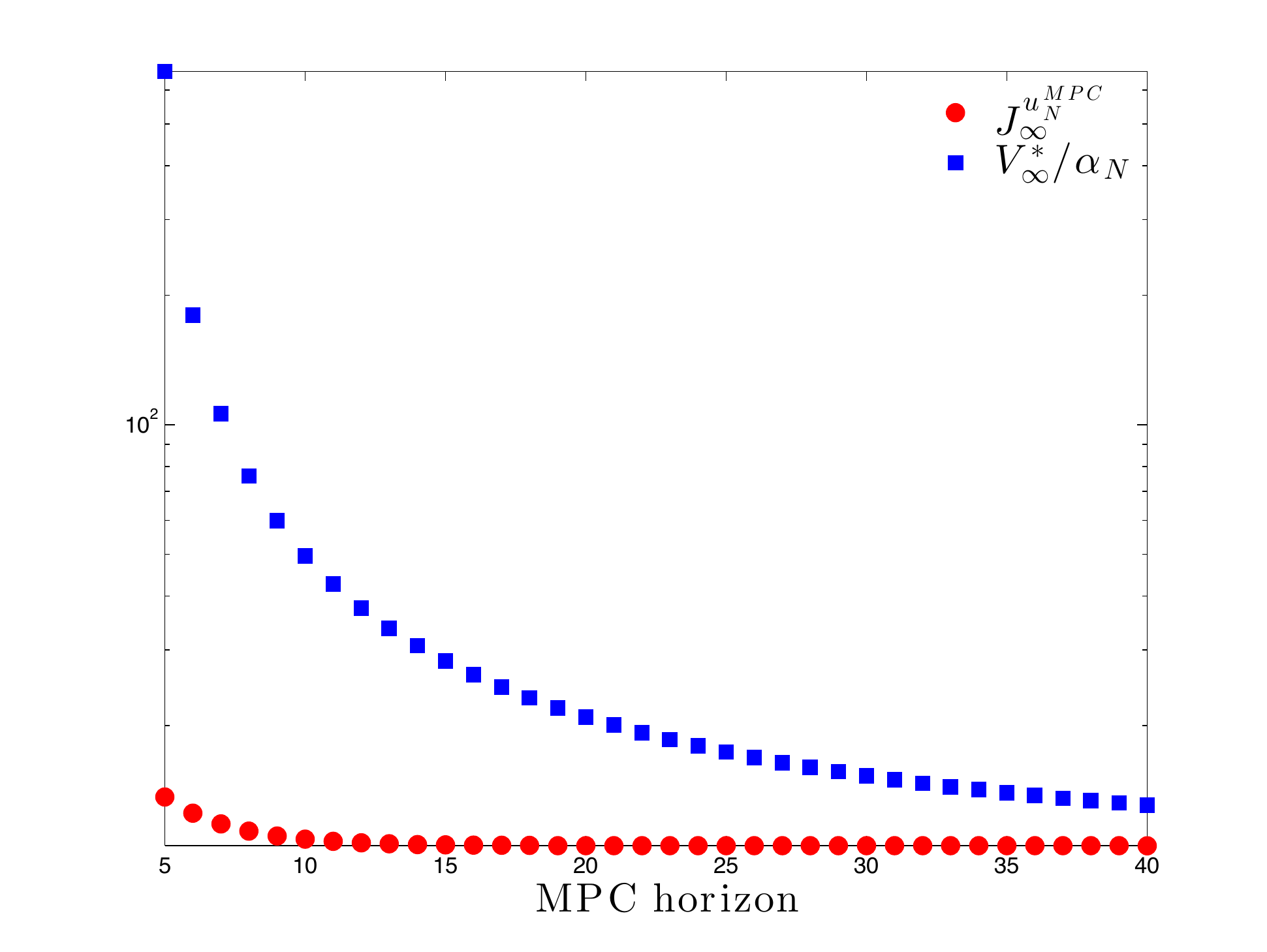}
\end{center}
\caption{Value of the cost functional $J_T^{u^{MPC}_N}(X_0)$ 
for controls obtained using a MPC strategy with control horizon $N$ (red) and presentation of the optimal costs  $V_T^{*}(X_0)$  multiplied by $\frac{1}{\alpha_N}$ where 
$\alpha_N$ is computed as in \cite[Theorem 5.4]{GrunePannekSeehafer2010aa}. For $N\leq 4$ no estimate of the type \eqref{eq:estimate} could be established.
}
\label{fig.01}
\end{figure}

We further investigate the behavior of the particle system \eqref{eq:dynamics_num} for controls with different MPC horizon. According
to the behavior of the  cost we expect that for increasing time horizon we are closer to the optimal cost. Defining $$E_n:=\int_{\R} x^{2} f(x) dx.$$
we obtain from equation \eqref{eq:dynamics_num}   
$$ E_{n+1} = - E_n + 2 Y_n^{2} + u_n Y_n.$$
The running cost tries to minimizes a trade--off of the mean of the distribution and the control action. If the mean $Y_n$ tends to zero,
then we observe that the energy $E_n$ tends to zero exponentially fast. Therefore, we expect with longer time horizon a mean $Y_n$ closer 
to zero and small variance of the solution to the kinetic equation. We simulate using $M=10^{5}$ discrete points randomly distributed on $\X=[-1,1]$ 
as initial condition $f_{M,0}$ as in equation \eqref{eq:idnuma}. The MPC control is computed according to the considerations above 
for $\nu=10^{2}$ and $\nu=10^3$ reported in Figure \ref{fig:NVar} and Figure \ref{fig:NVar2}. In both figures we show the computational results for the time 
evolution of the  distribution $f_n$ for $n=0,\dots, 100$. As expected longer optimization horizons leads to a faster decay in the variance of the distribution $f_n$. 

\begin{figure}
\centering
\subfigure[$N=2$]{\includegraphics[scale=0.2]{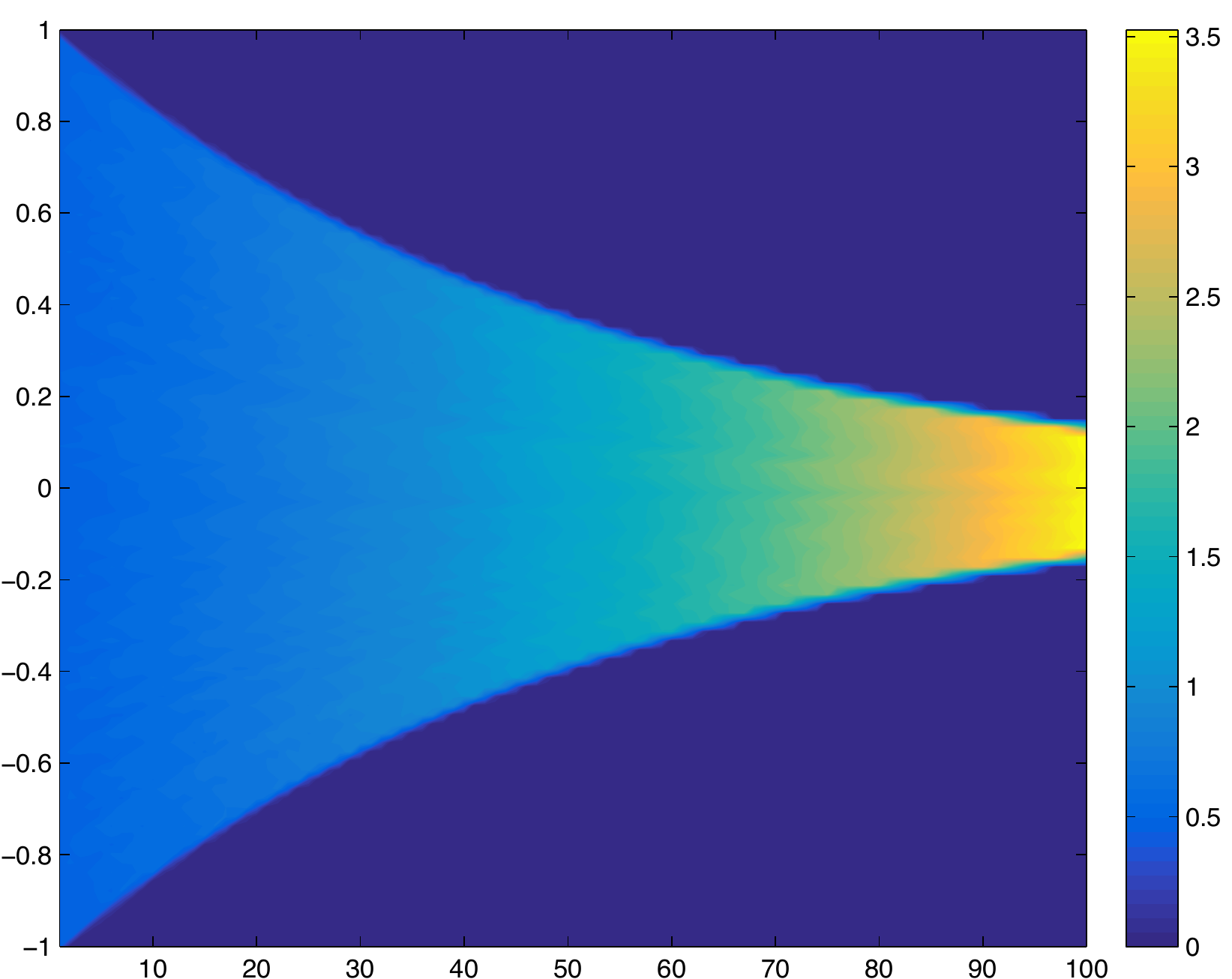}}
\subfigure[$N=3$]{\includegraphics[scale=0.2]{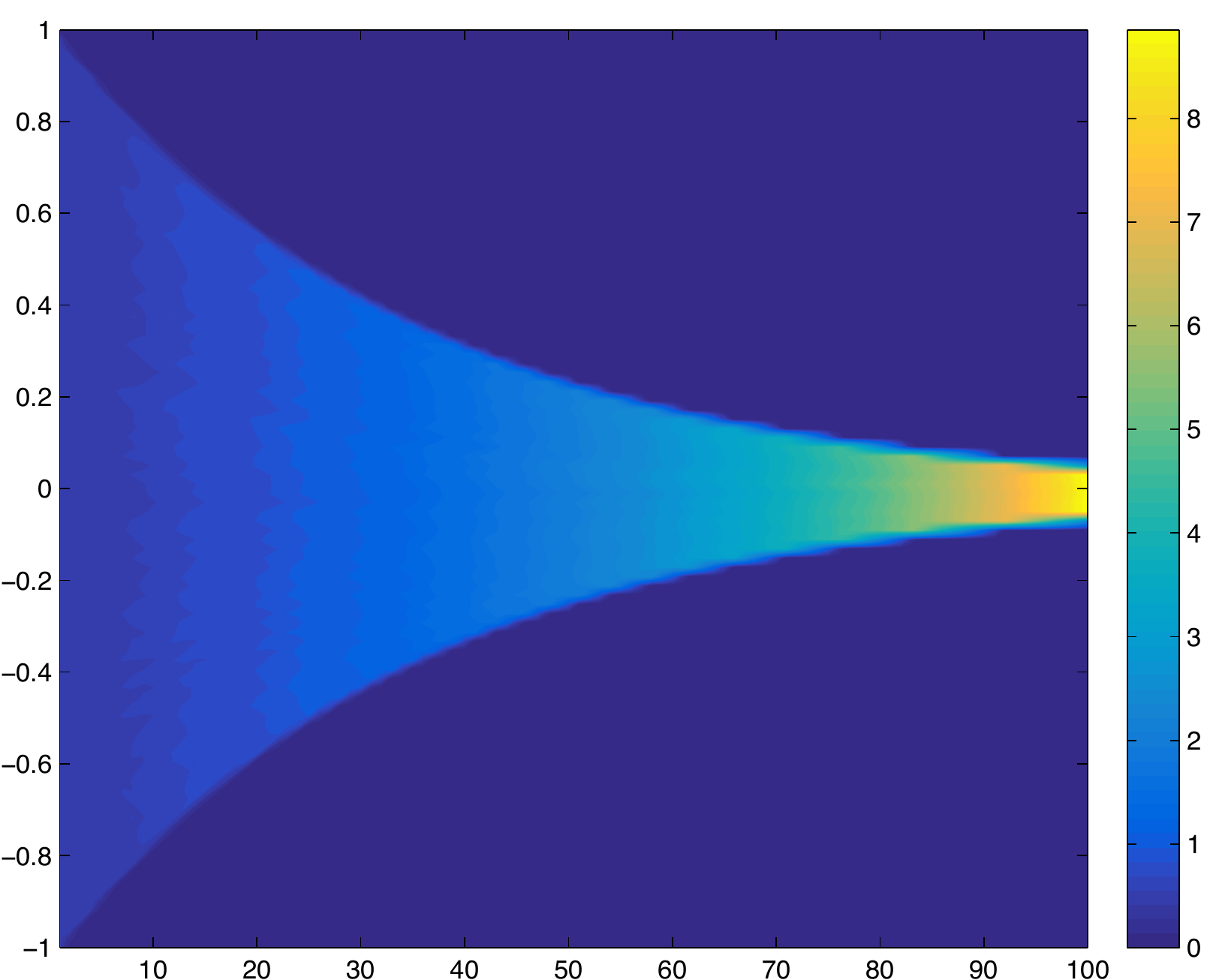}}
\subfigure[$N=4$]{\includegraphics[scale=0.2]{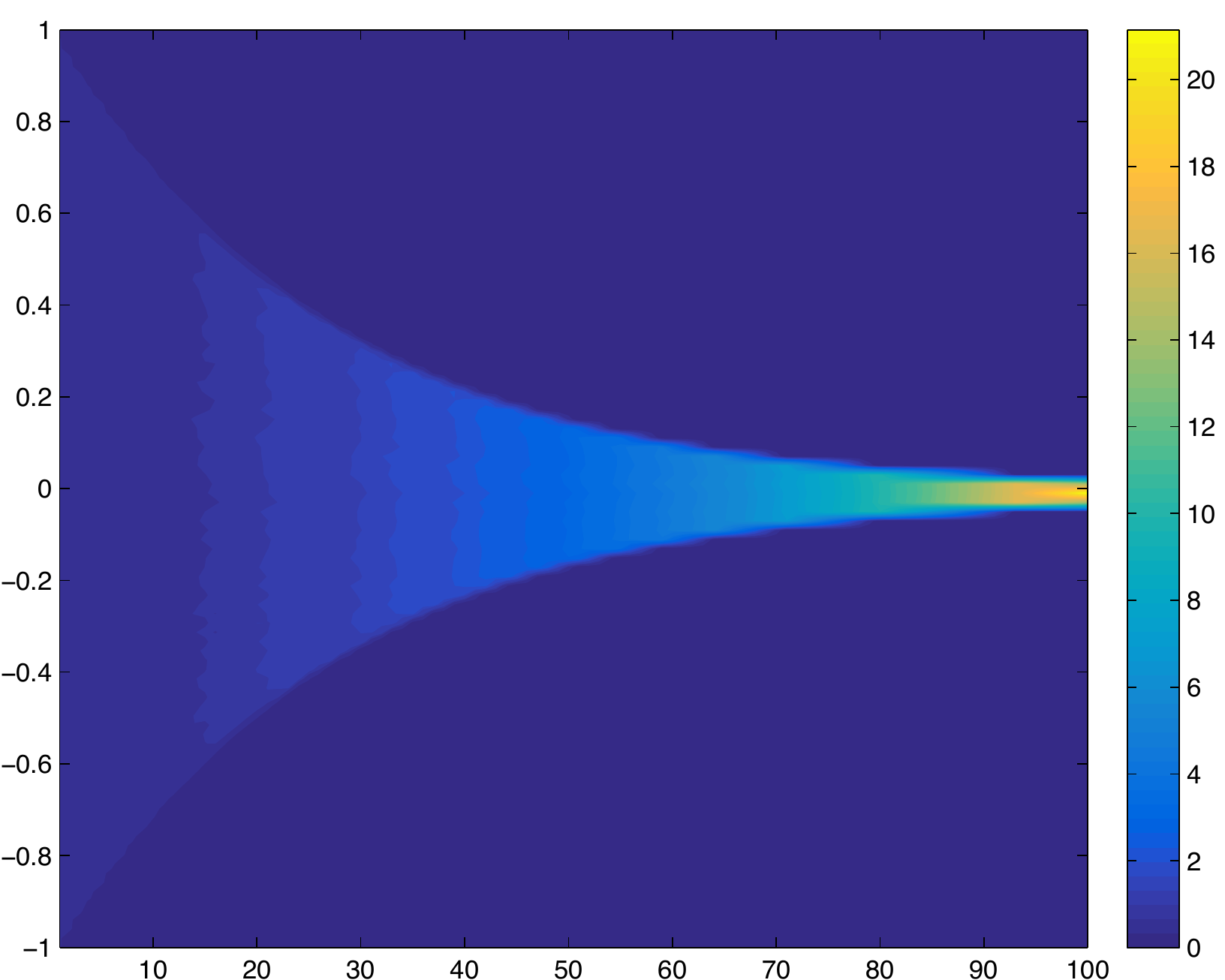}}\\
\subfigure[$N=5$]{\includegraphics[scale=0.2]{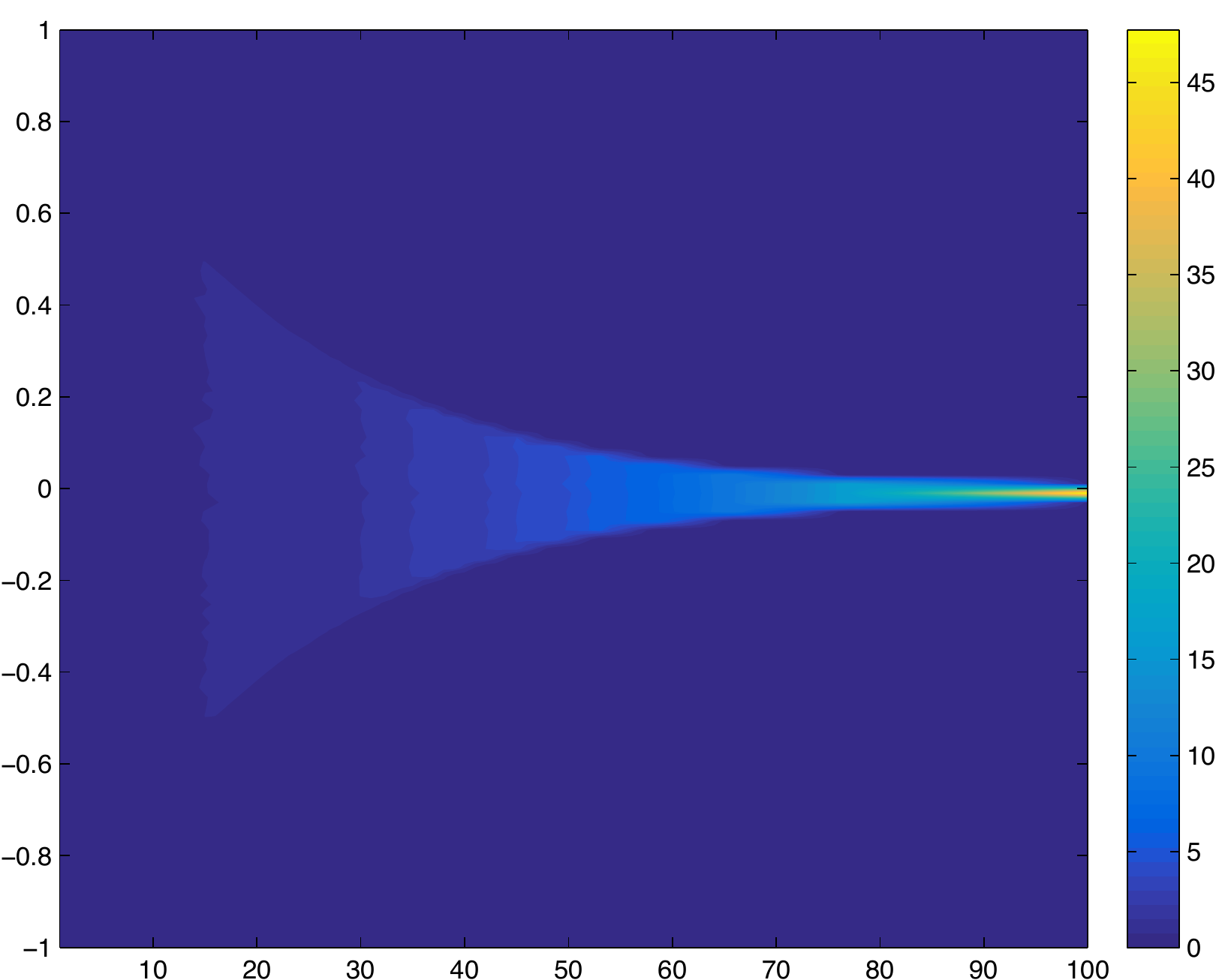}}
\subfigure[$N=6$]{\includegraphics[scale=0.2]{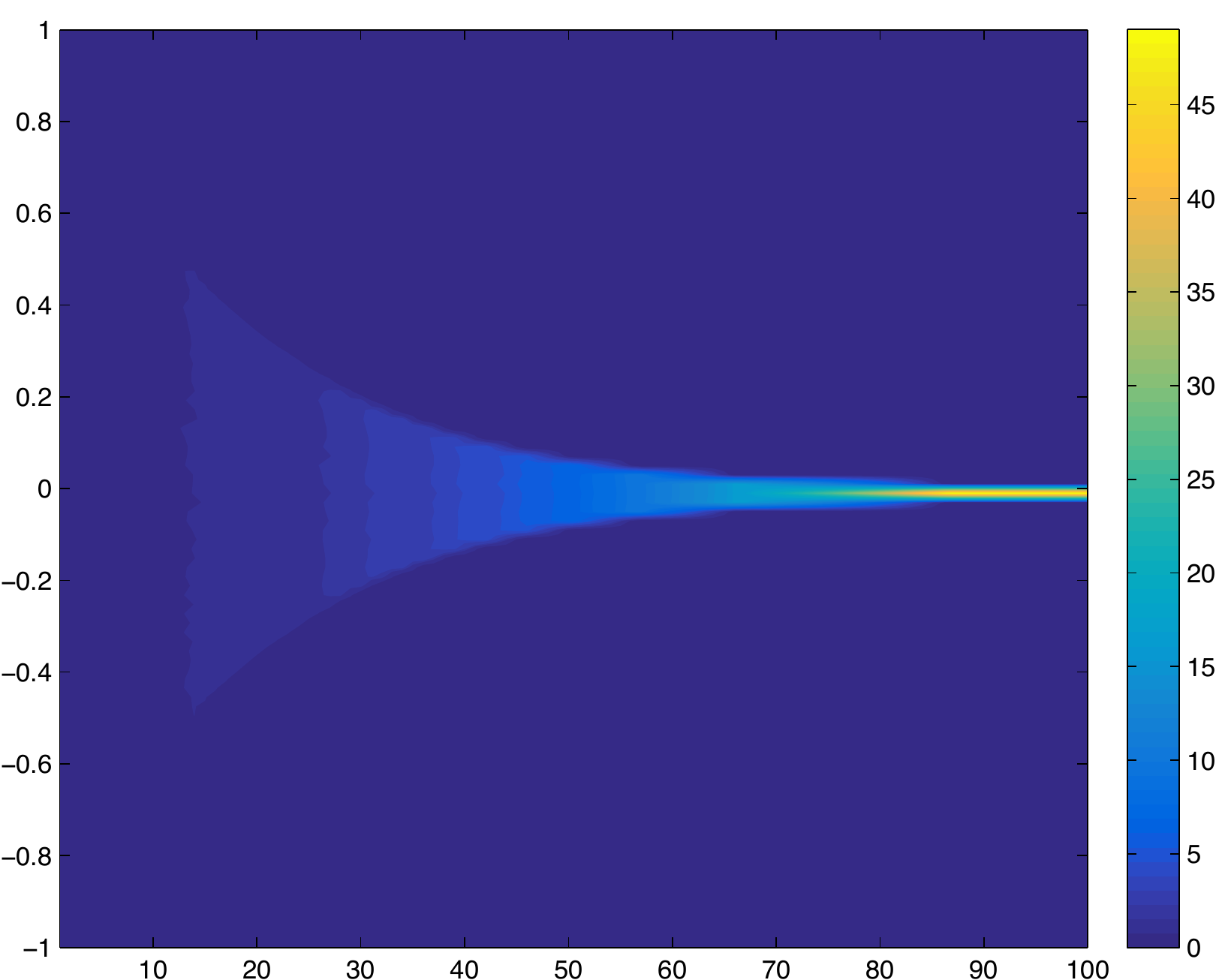}}
\subfigure[$N=7$]{\includegraphics[scale=0.2]{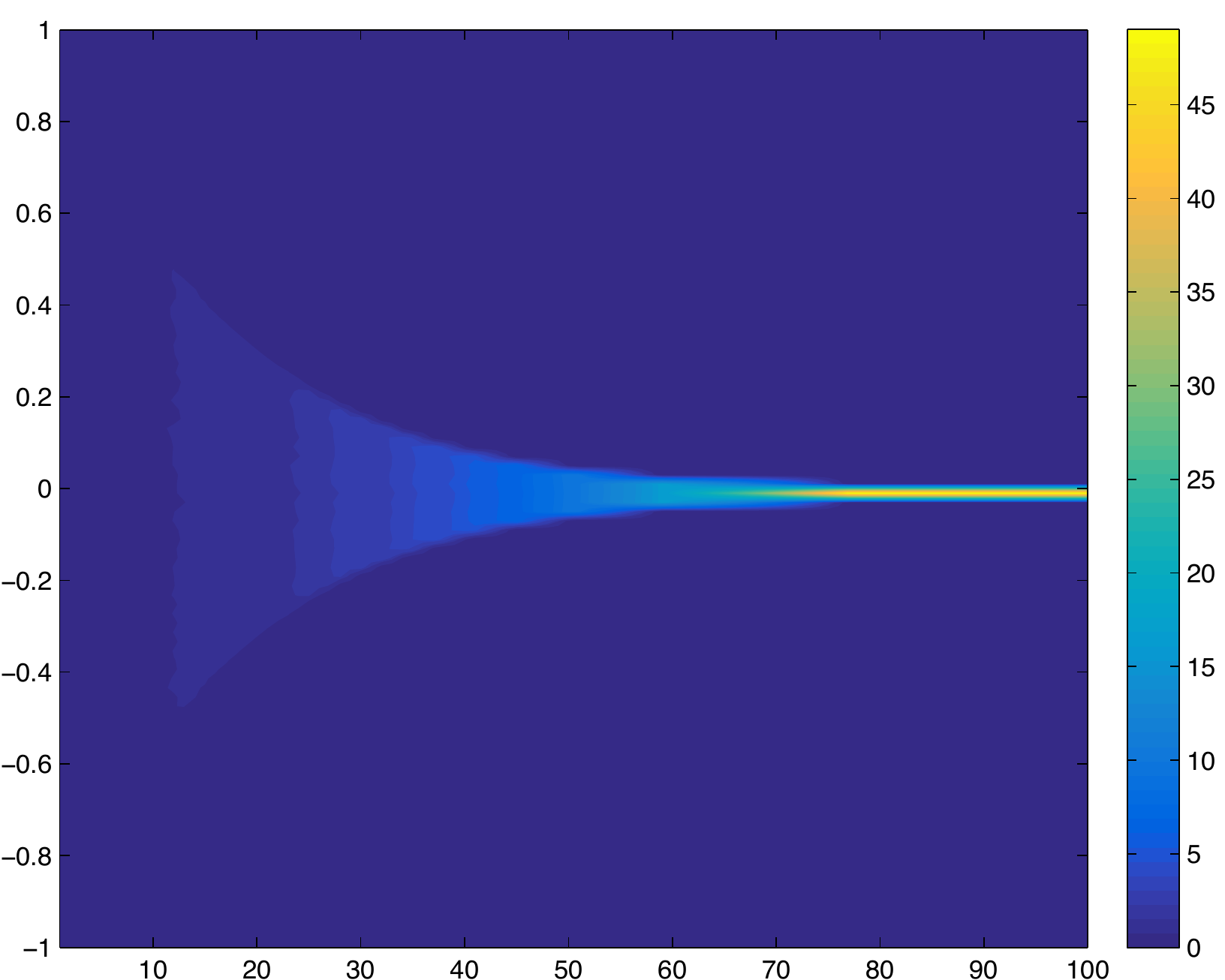}}\\
\subfigure[$N=8$]{\includegraphics[scale=0.2]{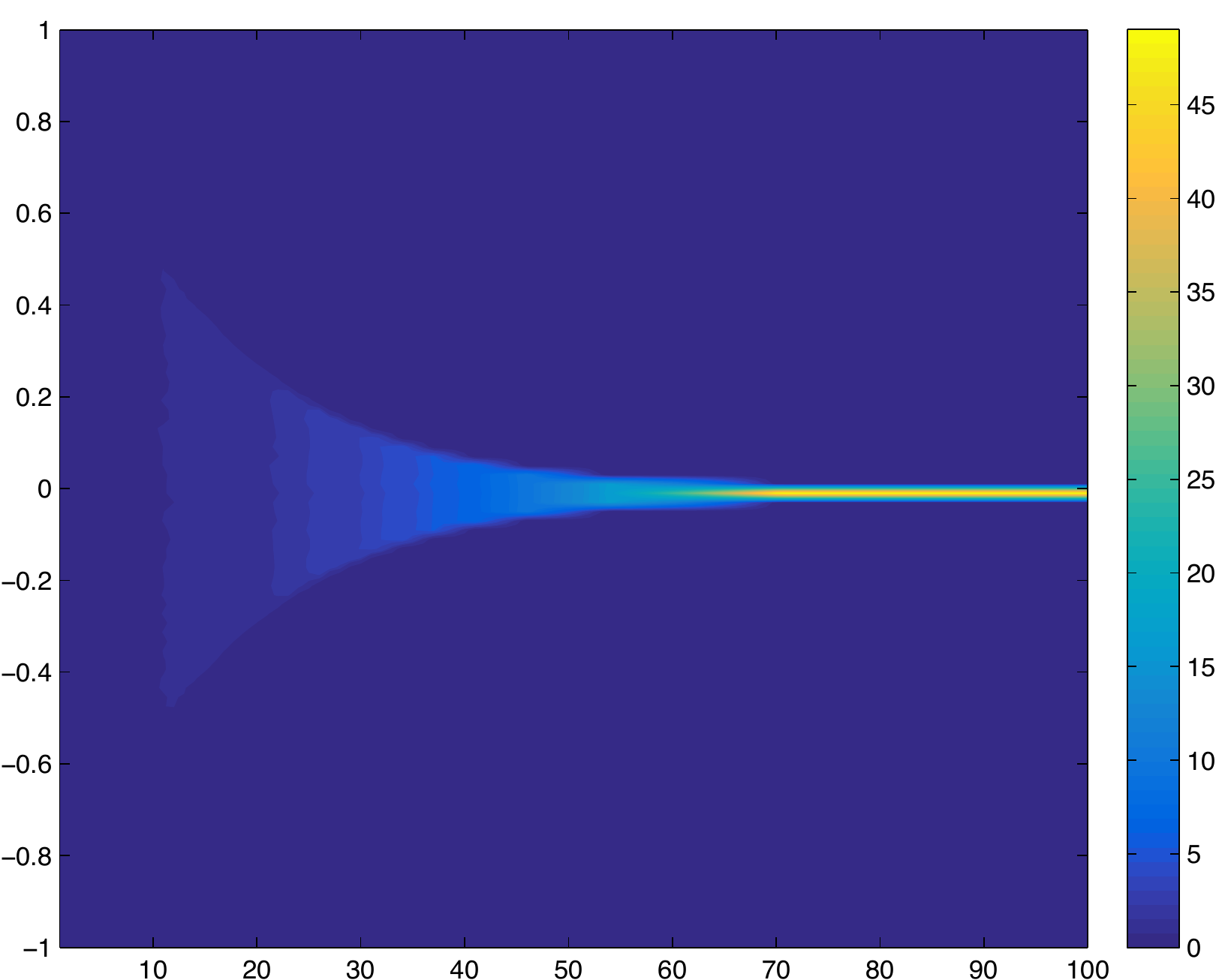}}
\subfigure[$N=9$]{\includegraphics[scale=0.2]{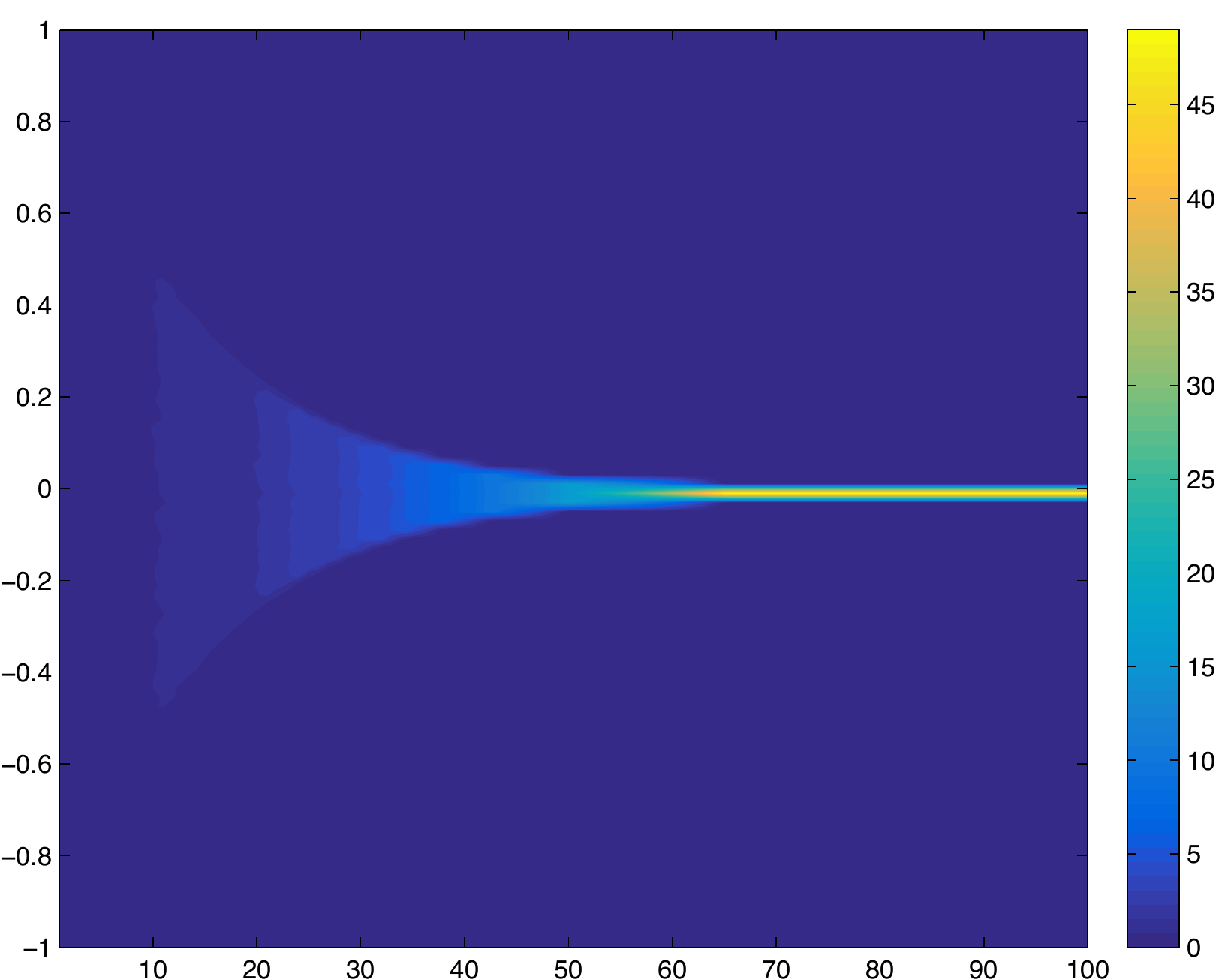}}
\subfigure[$N=10$]{\includegraphics[scale=0.2]{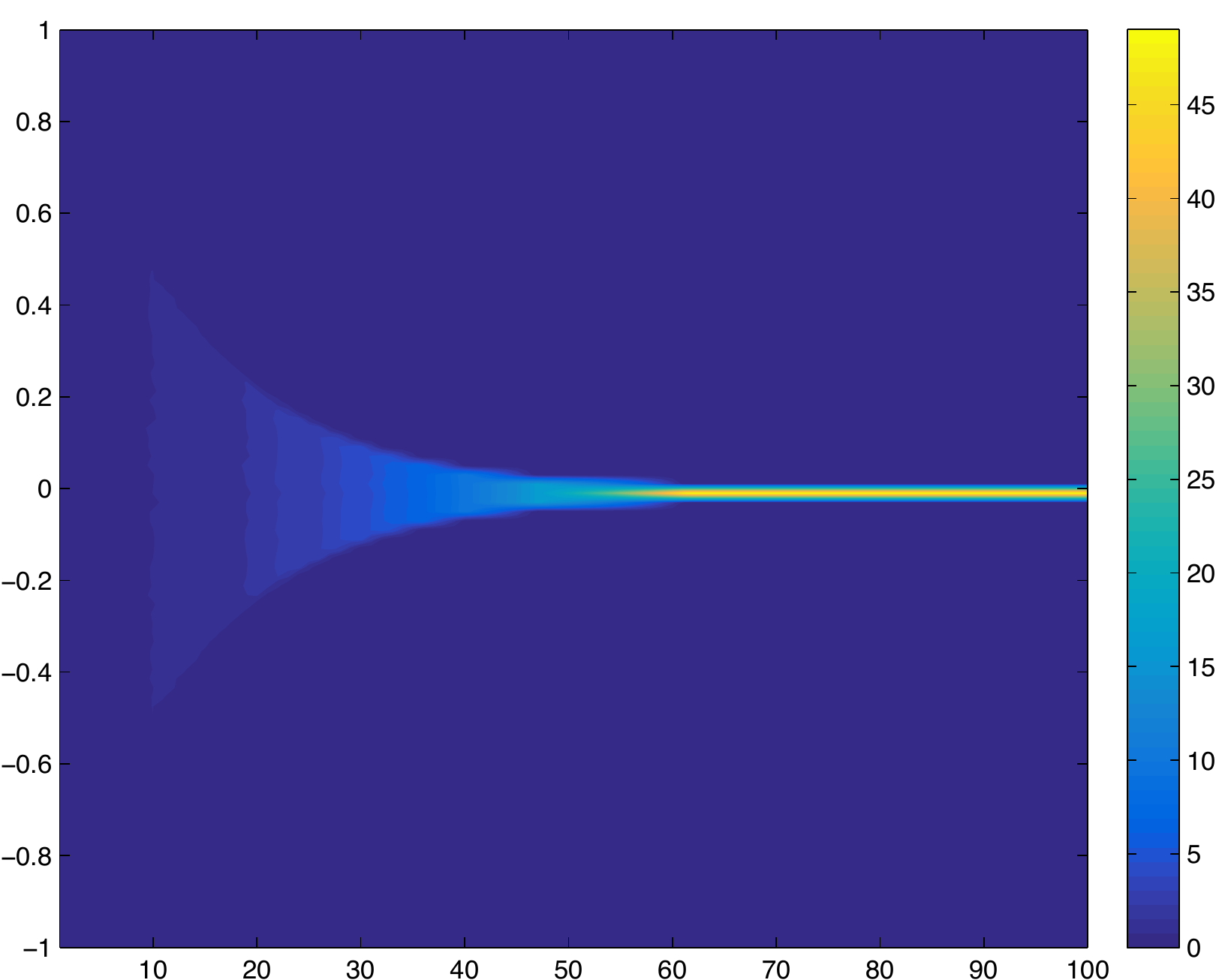}}
\caption{Experimental results for the optimization problem with varying optimization horizon $N$ and regularization constant $\nu=10^2$.} \label{fig:NVar}
\end{figure}

\begin{figure}
\centering
\subfigure[$N=2$]{\includegraphics[scale=0.2]{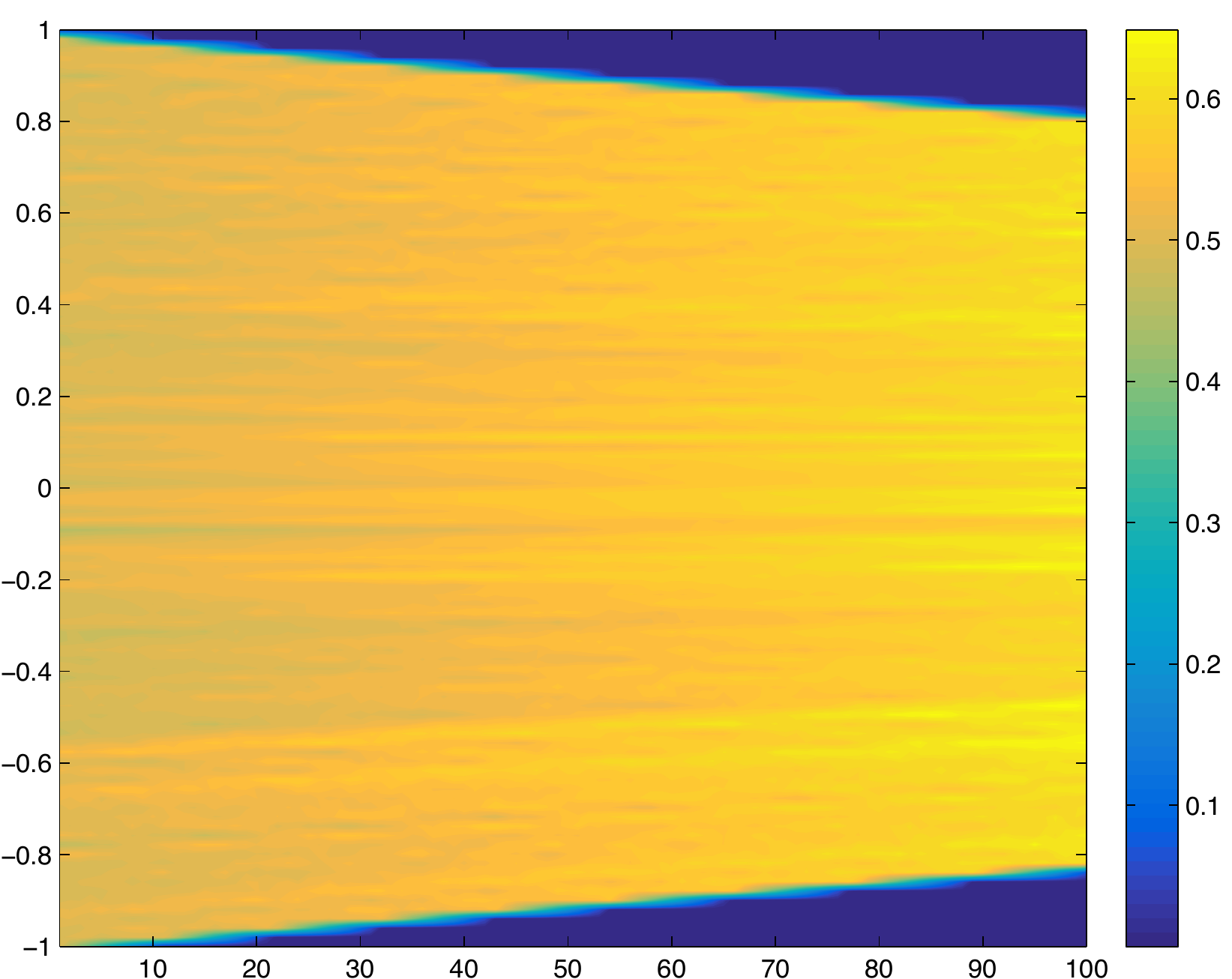}}
\subfigure[$N=3$]{\includegraphics[scale=0.2]{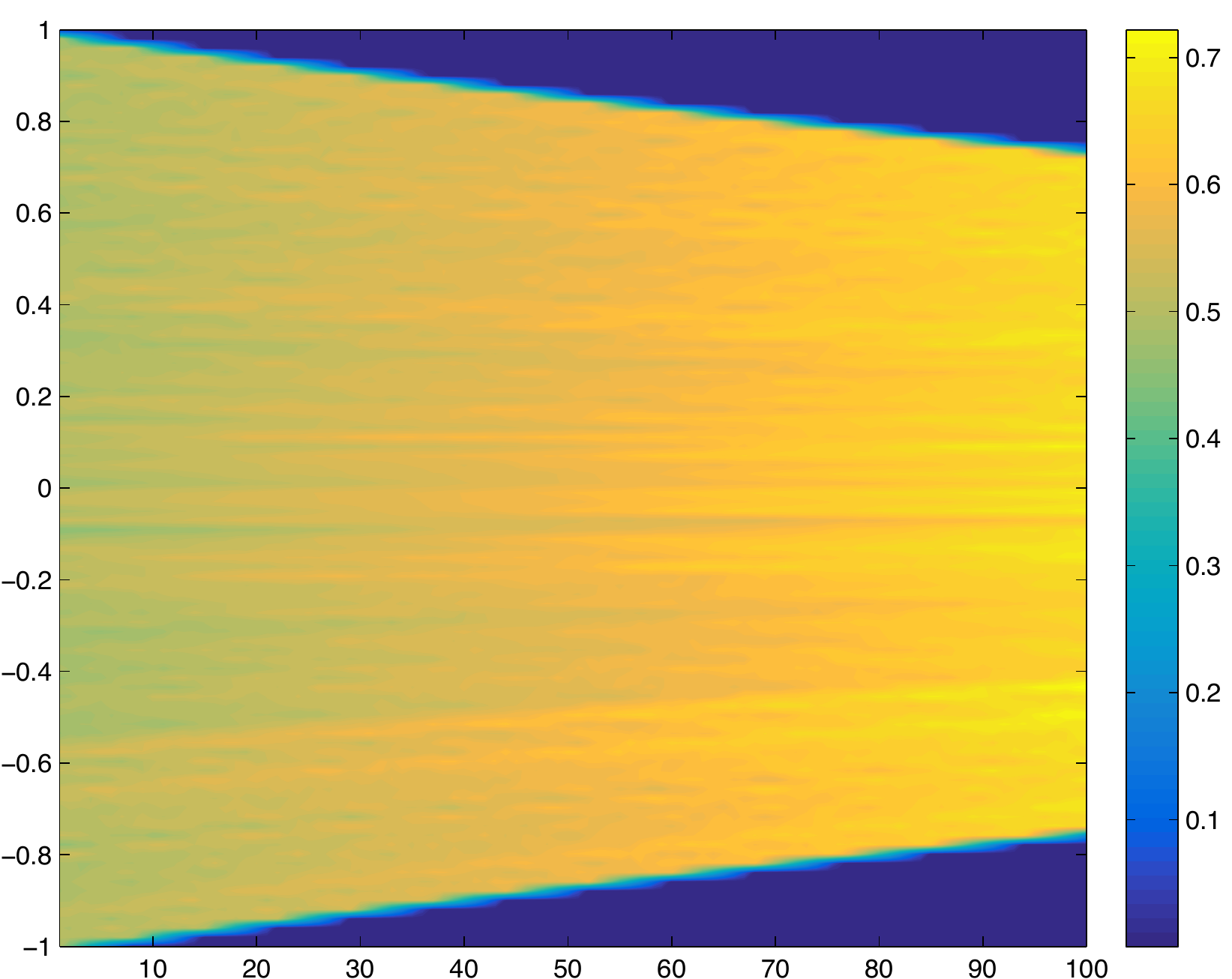}}
\subfigure[$N=4$]{\includegraphics[scale=0.2]{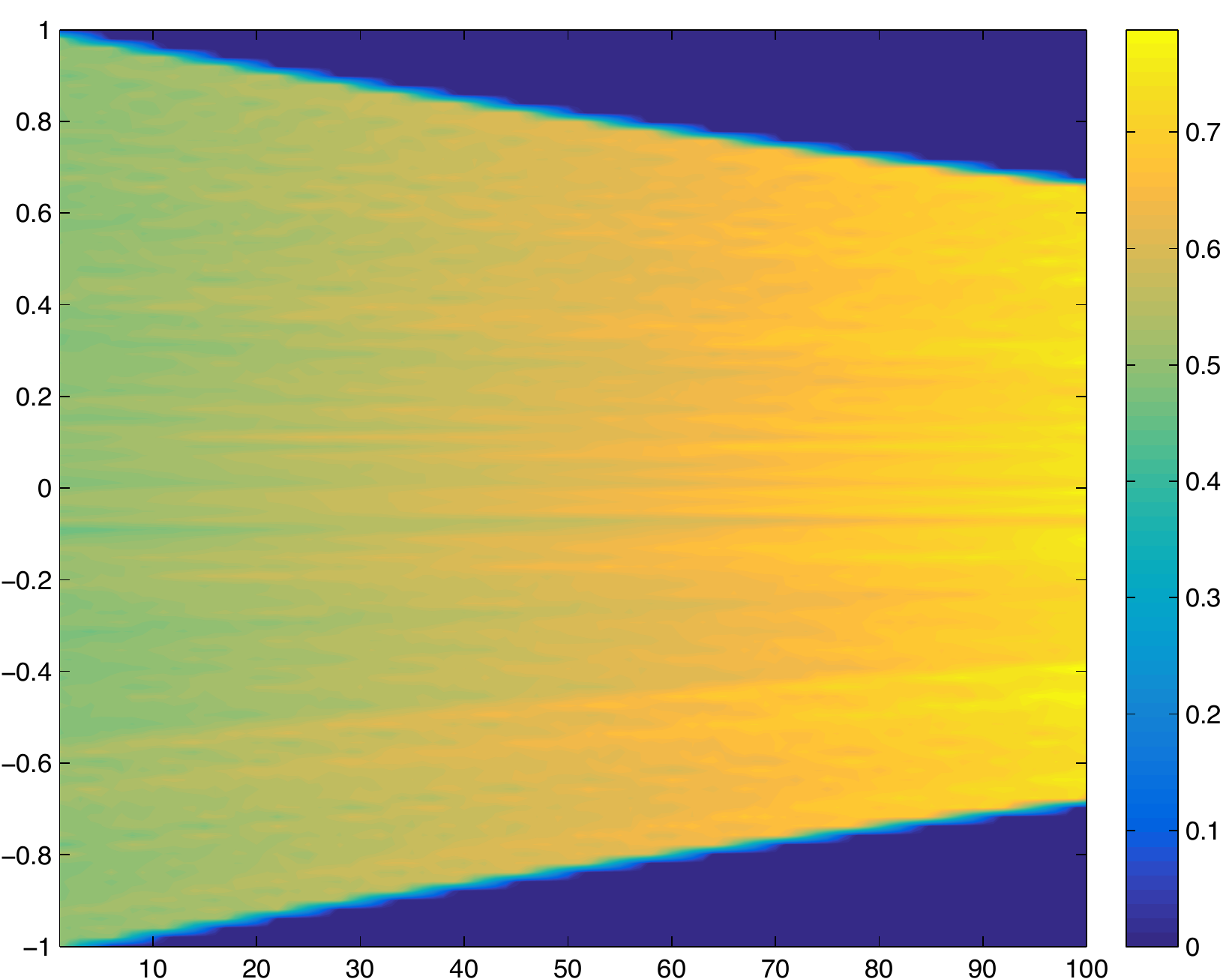}}\\
\subfigure[$N=5$]{\includegraphics[scale=0.2]{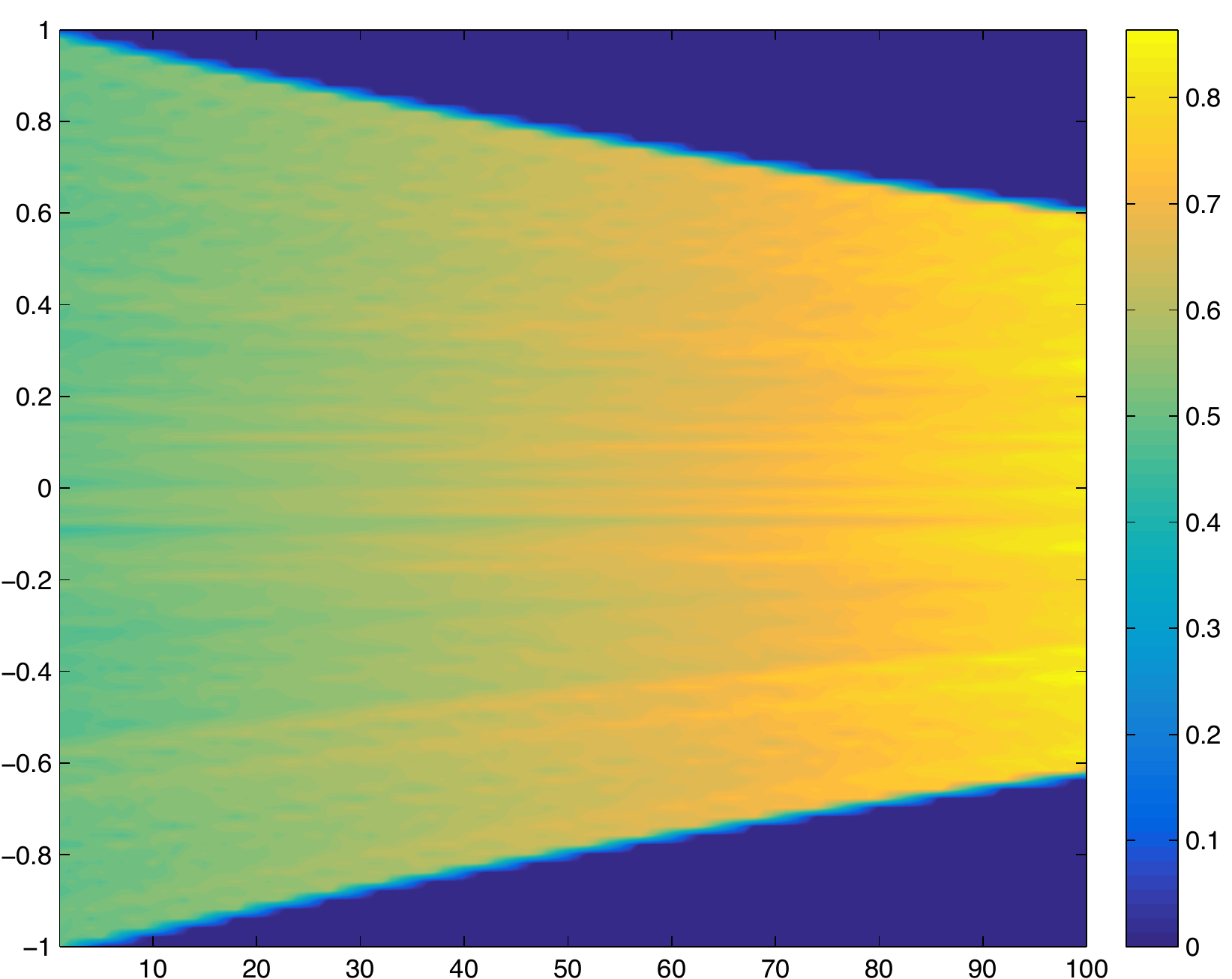}}
\subfigure[$N=6$]{\includegraphics[scale=0.2]{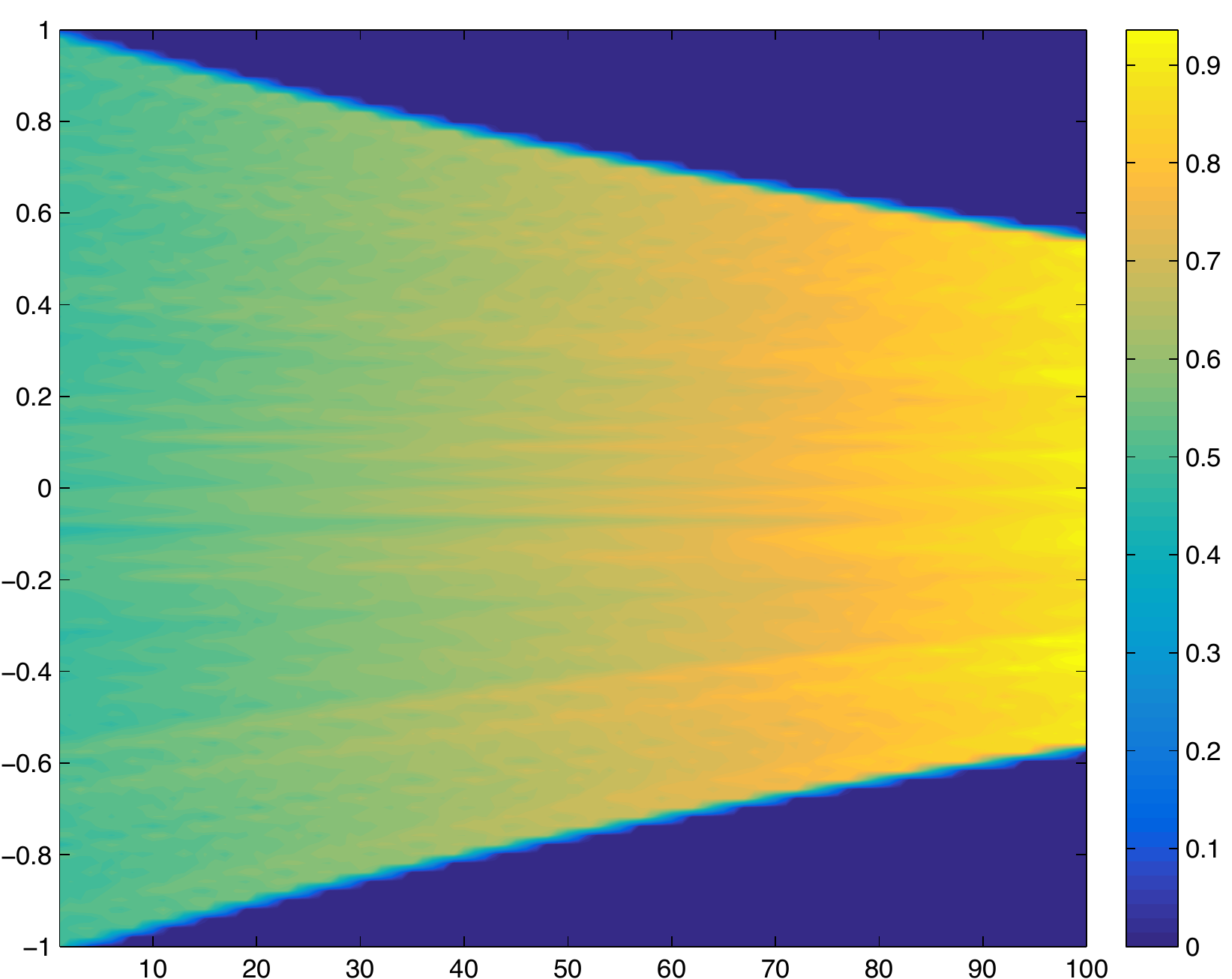}}
\subfigure[$N=7$]{\includegraphics[scale=0.2]{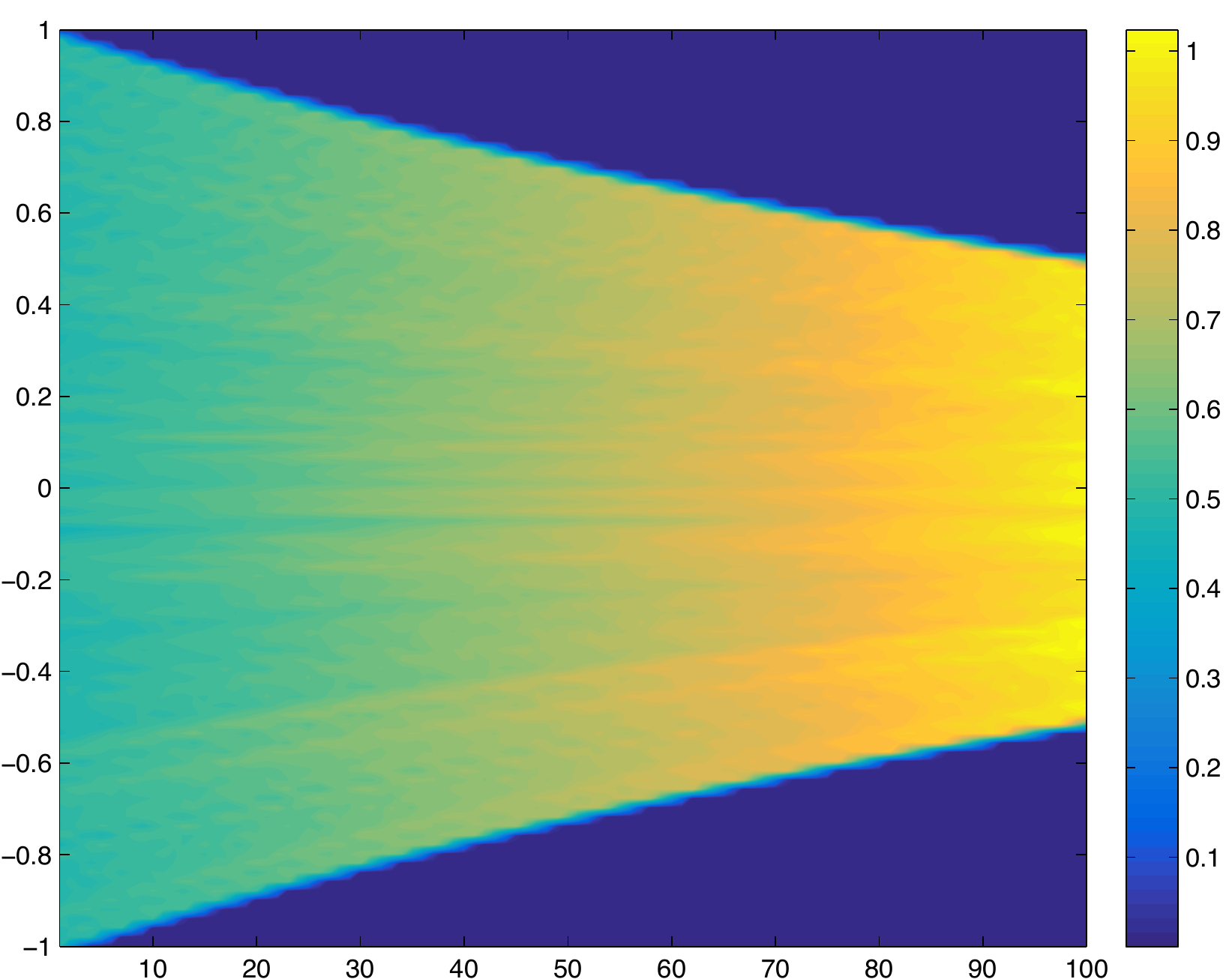}}\\
\subfigure[$N=8$]{\includegraphics[scale=0.2]{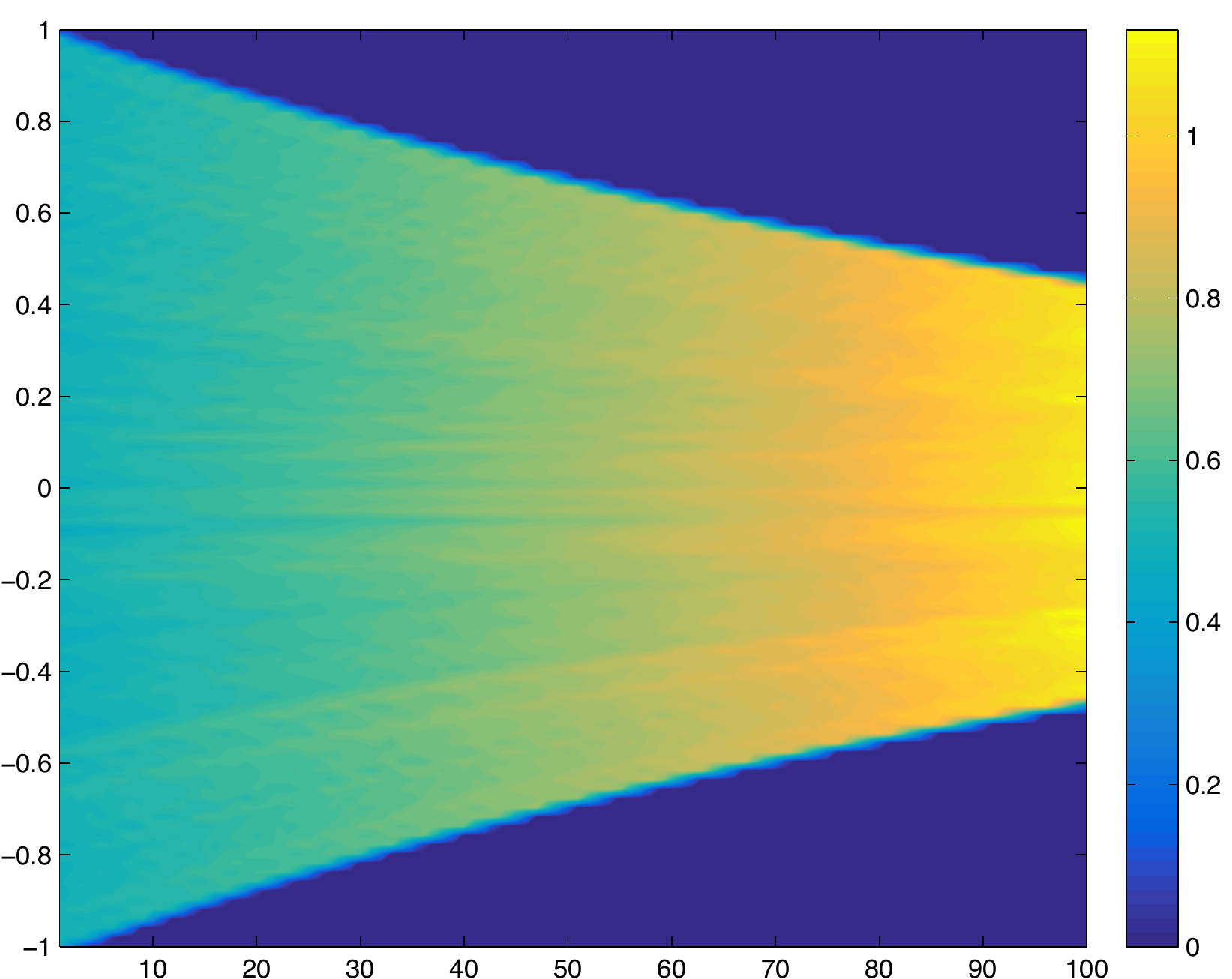}}
\subfigure[$N=9$]{\includegraphics[scale=0.2]{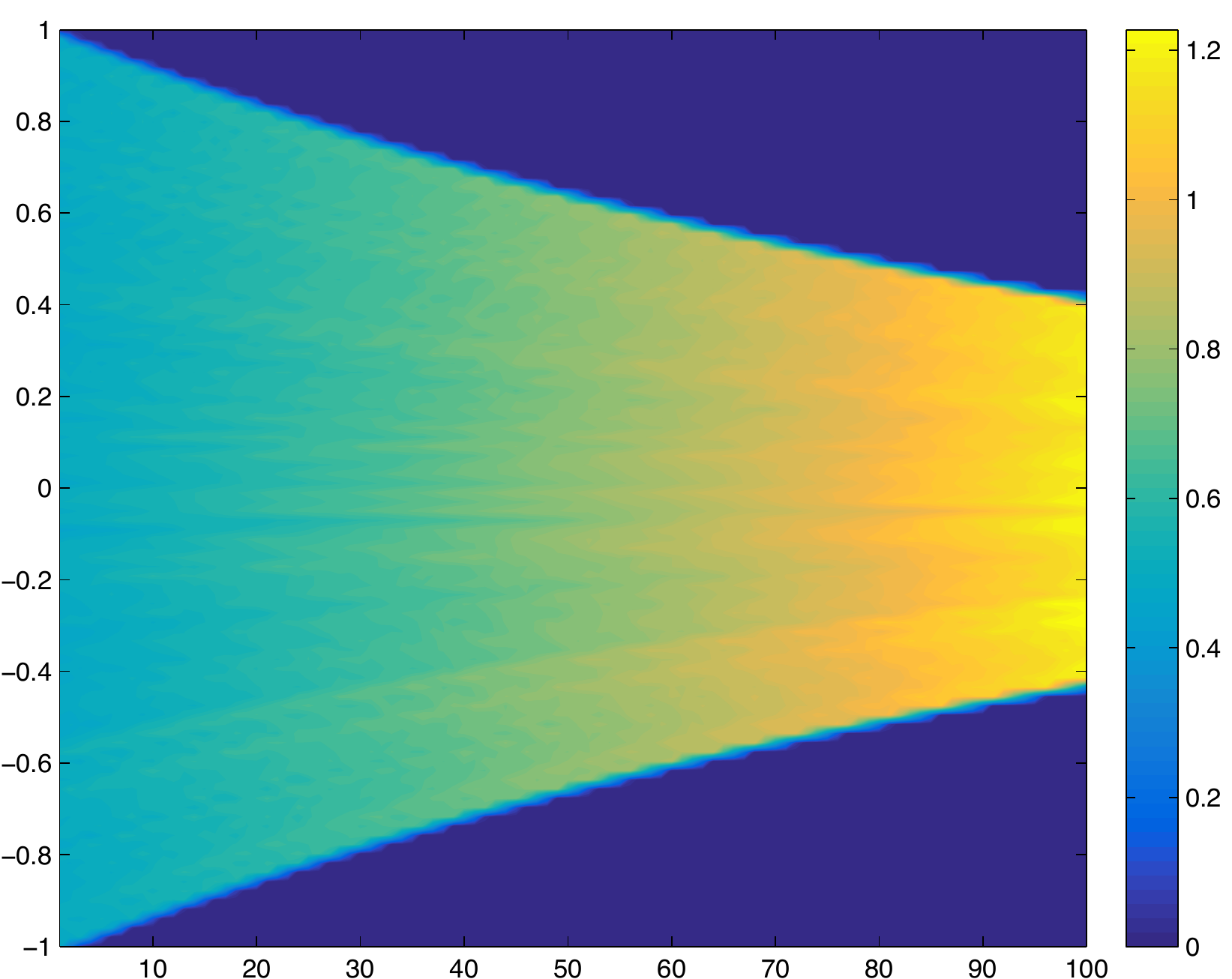}}
\subfigure[$N=10$]{\includegraphics[scale=0.2]{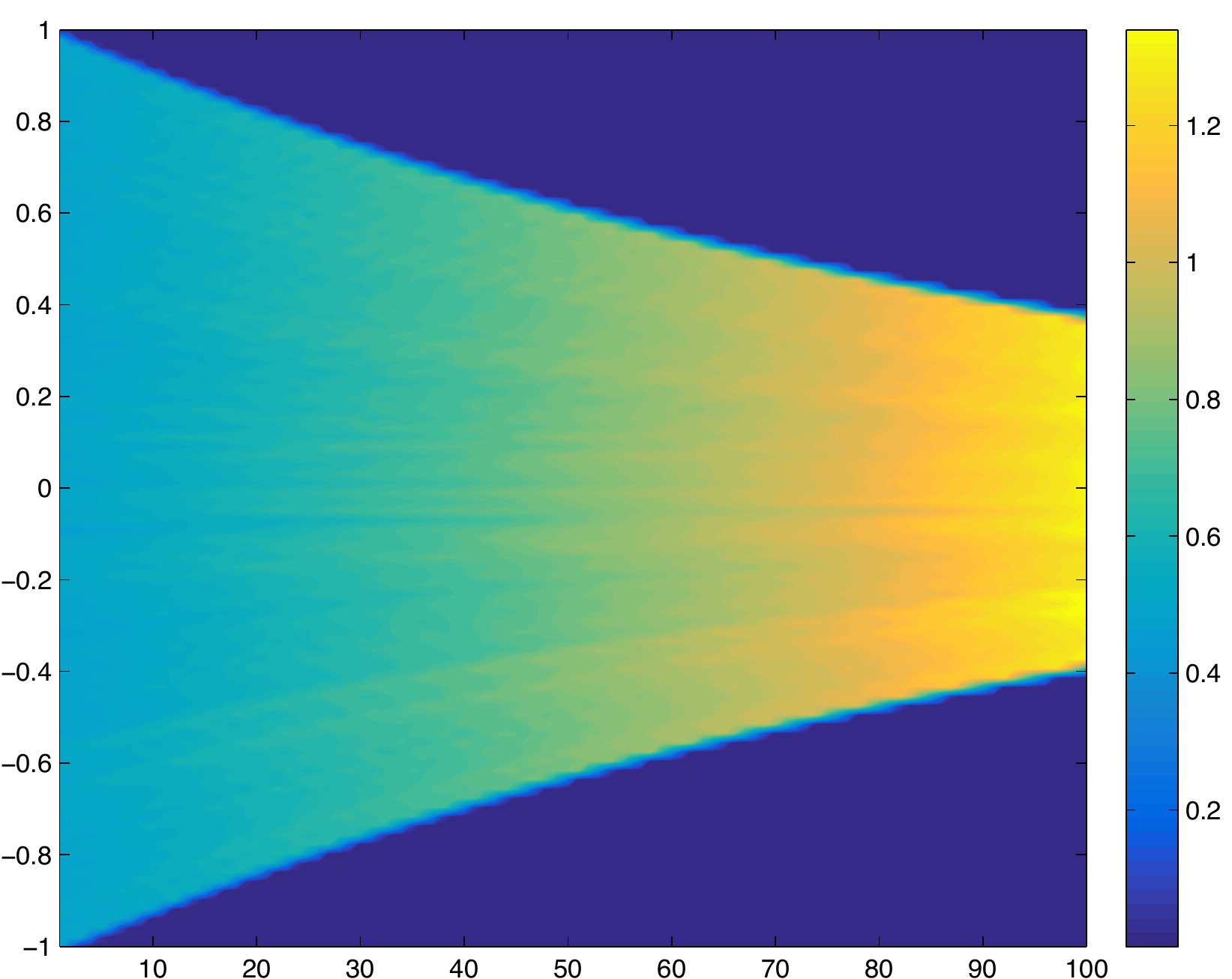}}
\caption{Experimental results for the optimization problem with varying optimization horizon $N$ and regularization constant $\nu=10^3$.}\label{fig:NVar2}
\end{figure}
%{\Large{MATTIA:} could you please re-run this result with $\nu=10^{3}$? you also need to recompute the optimal controls (with the code with $nu=1000$) if the result is interesting we can present it. } 
%

\section{Conclusion}
We have extended the estimates for the suboptimal MPC to the mean-field limit. The derived estimates
yield performance bounds for general symmetric multi--agent dynamics. Except for the assumptions necessary
to obtain the mean-field limit no additional requirements compared to the finite--dimensional theory
are required. The results apply to common agent dynamics modeling for example 
swarming, alignment and economics. We exemplified the theoretical results as well as the 
estimates on a simple opinion formation model.  The stability of the 
mean-field controller is still open and will be investigated in a forthcoming work. Further,
the estimates on $\alpha_N$ are pessimistic due to its generality. It is expected that the bounds can be improved for specific problems as in the finite dimensional case.

\section*{Acknowledgments}
This work has been supported DFG Cluster of Excellence Production Technologies for High--Wage Countries, STE2063/1-1,  and by the 5x1000 Grant provided by the University of Ferrara.

\appendix 
%\section{Appendix}
%\section{Technical details}%\label{app:A}
\section{}\label{app:A}
We collect some results of \cite{Cardaliaguet2010aa} for convenience; see also \cite[Theorem 4.1]{BlanchetCarlier2014aa}. 
The Kantorowich--Rubinstein distance $\done(\mu,\nu)$ for measures $\mu,\nu \in \mathcal{P}(Q)$ is given defined
by 
\begin{equation}
\label{def d1}
\done(\mu,\nu) := \sup \{ \int \phi\; d(\mu-\nu); \phi:Q\to \R, \phi \mbox{ is  1 - Lipschitz } \}. 
\end{equation}

\begin{theorem}[Theorem 2.1\cite{Cardaliaguet2010aa}]\label{Theorem2.1Card}
Let $Q^M$ be a compact subset of $\R^M$. 
Consider a sequence of functions $(u_M)_{M=1}^\infty$ with $u_M:Q^{M} \to \R.$  
Assume each $u_M(X)=u_M(x_1,\dots,x_M)$ is a symmetric function in all variables, i.e., 
$$ u_M(X)=u_M( x_{\sigma(1)}, \dots, x_{\sigma(M)} )$$
for any permutation $\sigma$ on $\{1,\dots,M\}.$ Denote by
$\done$ the Kantorowich--Rubenstein distance on the space
of probability measures $\mathcal{P}(Q)$ and let $\omega$ 
be a modulus of continuity independent of $M$. Assume that 
the sequence is uniformly bounded $\| u_M \|_{L^\infty(Q^M)} \leq C$. Further assume that 
for all $X,Y \in Q^M$ and all $M$  we have
$$ | u_M(X)-u_M(Y) | \leq \omega( {\bf d}_1( m^M_X, m^M_Y) )$$
where $m^M_\xi \in \mathcal{P}(Q)$ is defined by $m^M_\xi(x) = \frac{1}M \sum\limits_{i=1}^M \delta(x-\xi_i)$.
\par 
Then there exists a subsequence $(u_{M_k})_k$ of $(u_M)_M$ and a continuous map 
$U:\mathcal{P}(Q) \to \R$ such that 
\begin{equation}
\label{eq:conv sense}
\lim\limits_{k\to \infty} \sup\limits_{ X\in \R^M}  | u_{M_k}(X) - U( m^{M_k}_X) | =0.
\end{equation}
\end{theorem}

Theorem \ref{Theorem2.1Card} has been extended to the case of functions
$g(x_i,X_{-i}): \mathcal{X}^M \subset \R^M \to \R$ being symmetric only in $X_{-i}.$ 
Here, $\mathcal{X}$ is a compact subset of $\R.$  The corresponding
result is given in \cite[Section 4]{BlanchetCarlier2014aa} and repeated here for convenience.
For any permutation $\sigma$ of the set $\{1,\dots,M\} \backslash \{i\}$ 
and all $x_i \in \R$ we have 
$g(x_i,X_{-i}) = g(x_i, (x_{\sigma(j)})_{j\not = i} ).$ 
Moreover, there exists a modulus of continuity $\omega$ such that for all $x_i,y_i\in \R$ 
and all $M$ we have
$$ \| g(x_i,X_{-i}) - g(y_i,Y_{-i}) \| \leq \omega( \| x_i-y_i\|) +  \omega( {\bf d}_1( m^{M-1}_{X_{-i}}, m^{M-1}_{Y_{-i}}) ).$$
Further assume that $\| g(X)\|_{L^{\infty}(\R^M)} \leq C$. Then, $g(x_i,X_{-i}): \R^{M}\to\R$ can be extended 
to a function $G_M:\mathcal{X} \times \mathcal{P}(\mathcal{X}) \to \R$ by 
\begin{equation}
\label{def-glimit}
G_M(x,\nu) = \inf\limits_{X_{-i} \in \R^{M-1}} \{ g(x,X_{-i}) +  \omega( {\bf d}_1( m^{M-1}_{X_{-i}}, \nu) ) \}.
\end{equation}
It can be shown as before that $(G_M)_M$ is a sequence of uniformly equi--continuous functions
on $\mathcal{X} \times \mathcal{P}(\mathcal{X}).$ Therefore, $(G_M)_M$ converges to 
 a function $G:\mathcal{X} \times \mathcal{P}(\mathcal{X}),$ see also \cite[Theorem 4.1]{BlanchetCarlier2014aa}.

%\section{mean-field limit of alignment model } \label{sec:mean-field}
\section{} \label{sec:mean-field}
In this section we derive a semi discrete mean-field formulation of the constrained problem \eqref{eq:mpc00}. Let us suppose that the introduced control $u=u^{MPC}_n$ is symmetric with respect to each position of the system of agents at time $t^n$. We define the empirical measures 
\begin{equation}
f_M(t^n)=f_M^n = \dfrac{1}{M}\sum_{i=1}^M \delta(x-x_{i,n}),
\end{equation}
where $\delta $ is the Dirac delta, or localizing function, defined in the space of probability measures of $\R^d$, namely $\mathcal{P}(\R^d)$. For any smooth function $\phi\in\mathcal{C}_0^1(\mathbb{R}^d)$ we have
\begin{equation}
\int_{\RR} \phi(x)f^n_M(x)dx=\dfrac{1}{M}\sum_{i=1}^M\phi(x_{i,n}),
\end{equation}
then through a first order Taylor expansion we obtain
\begin{equation}\label{eq:taylor}
\phi(x_{i,n+1})-\phi(x_{i,n})=\phi’(x_{i,n})(x_{i,n+1}-x_{i,n})+O(\Delta t^2)
\end{equation}
Now from the original dynamic \eqref{eq:mpc00} we can replace the quantity $x_{i,n+1}-x_{i,n}$ in \eqref{eq:taylor} obtaining
\begin{equation}
\phi’(x_{i,n})\left[\dfrac{\Delta t}{M}\sum_{j=1}^M P(x_{j,n}-x_{i,n})+\Delta t u_n^{MPC}\right]
\end{equation}
and summing up to $M$ we have
\begin{equation}\label{eq:preback}
\dfrac{1}{M}\sum_{i=1}^M  \phi(x_{i,n+1})-\phi(x_{i,n})=\dfrac{1}{M}\sum_{i=1}^M \phi’(x_{i,n})\left[\dfrac{\Delta t}{M}\sum_{j=1}^M P(x_{j,n}-x_{i,n})+\Delta t u_n^{MPC}\right].
\end{equation}
%In an unconstrained setting, in order to converge to a distribution function in the particle limit $M\rightarrow +\infty$, we need to place additional hypothesis on the initial empirical measure, \cite{CFTV}. First, we start supposing that the agents remain in a fixed and compact domain for each $t^n,n=1,\dots,N$ that is $f_M^n$ is tight in $\mathcal{P}(\RR^d)$, where $\RR^d$ is a Polish space (in our case it is something more, it is Euclidean). 
Given that $f_M^n$ is a probability measure in the space $\mathcal{P}(\RR^d)$ with uniform support with respect to $M$, Prokhorov’s theorem implies that the sequence $( f^n_M)_{M}$ is weakly-* relatively compact, i.e. there exists a subsequence $\left(f^n_{M_m}\right)_m$ and a probability measure $f^n\in \mathcal{P}(\RR^d)$ such that
\begin{equation}\label{eq:w*}
f^n_{M_m} \rightarrow_{w*} f^n
\end{equation}
in $\mathcal{P}(\RR^d)$. Recall that for the Cucker-Smale model the tightness hypothesis is in general satisfied if the initial distribution $f^0_M$ is compactly supported with respect to $M$. For a rigorous proof we refer to \cite{CarrilloFornasierToscani2010aa,DiFrancescoRosiniArchive2015aa}. \\

%\begin{thebibliography}{00}
%\bibliographystyle{00}
%\bibliographystyle{siam}
%\bibliography{completeBibTex}

\end{document}